\documentclass[11pt]{article}
\usepackage[T1]{fontenc}
\usepackage{amsmath}%
\usepackage{amsthm}
\usepackage{color,hyperref}
\usepackage{amsxtra}%
\usepackage{enumitem}
\usepackage{amsfonts}%
\usepackage{amssymb}%
\usepackage{graphicx}
\usepackage{subfigure}
\usepackage[margin=3cm]{geometry}
\usepackage{cite}
\usepackage{array}
\usepackage{booktabs}
\newcommand{\lsc}{\text{LSC}}
\newcommand{\usc}{\text{USC}}
\renewcommand{\bar}[1]{{\overline{#1}}}
\newcommand{\F}{{\mathcal F}}

\newcommand{\R}{\mathbb{R}}

\newcommand{\Z}{\mathbb{Z}}
\newcommand{\vb}[1]{\mathbf{#1}}

\newcommand{\eps}{\varepsilon}

\renewcommand{\tilde}[1]{\widetilde{#1}}
\renewcommand{\phi}{\varphi}

\renewcommand{\t}[1]{\times 10^{-#1}}

\renewcommand{\subset}{\subseteq}

\theoremstyle{plain}
\newtheorem{theorem}{Theorem}
\newtheorem*{theorem*}{Theorem}

\newtheorem{corollary}[theorem]{Corollary}
\newtheorem{lemma}[theorem]{Lemma}
\newtheorem{proposition}[theorem]{Proposition}

\theoremstyle{definition}

\newtheorem{definition}[theorem]{Definition}

\theoremstyle{remark}

\newtheorem{remark}[theorem]{Remark}
\numberwithin{equation}{section}
\numberwithin{theorem}{section}
\graphicspath{{images/}}

\begin{document}

\title{Numerical schemes and rates of convergence for the Hamilton-Jacobi equation continuum limit of nondominated sorting\thanks{The research described in this paper was partially supported by NSF grants DMS-1500829 and DMS-0914567. Part of this work was completed while the author was supported by a Rackham Predoctoral Fellowship.}}
%\keywords{Hamilton-Jacobi equations \and viscosity solutions \and numerical schemes \and rate of convergence \and nondominated sorting \and longest chain problem \and multiobjective optimization }
%\subjclass[2010]{35D40 \and 35F21 \and 65N06}

\author{Jeff Calder\thanks{Department of Mathematics, University of California, Berkeley. ({\tt jcalder@berkeley.edu})}}
\maketitle
\begin{abstract}
Non-dominated sorting arranges a set of points in $n$-dimensional Euclidean space into layers by repeatedly removing the coordinatewise minimal elements. It was recently shown that nondominated sorting of random points has a Hamilton-Jacobi equation continuum limit. The obvious numerical scheme for this PDE has a slow convergence rate of $O(h^\frac{1}{n})$. In this paper, we introduce two new numerical schemes that have formal rates of $O(h)$ and we prove the usual $O(\sqrt{h})$ theoretical rates. We also present the results of numerical simulations illustrating the difference between the formal and theoretical rates. 
\end{abstract}

\section{Introduction}
\label{sec:intro}

In this paper, we introduce new finite difference schemes for the Hamilton-Jacobi equation
\begin{equation}\label{eq:Pintro}
\left.
\begin{aligned}
u_{x_1} \cdots u_{x_n} &= f& &\text{in } \R^n_+\\
u &=0& &\text{on } \partial\R^n_+, 
\end{aligned}\right\}
\end{equation}
and prove rates of convergence. 

The Hamilton-Jacobi equation \eqref{eq:Pintro} appeared recently as the continuum limit of nondominated sorting, which is widely used in scientific and engineering contexts~\cite{calder2014}. Let us briefly describe the connection. Let $X_1,\dots,X_N$ be \emph{i.i.d.}~random variables on $\R^n_+$ with continuous density $f$. Let $\F_1$ denote the elements in $S:=\{X_1,\dots,X_N\}$ that are coordinatewise minimal. The set $\F_1$ is called the \emph{first Pareto front} of $S$, and the elements of $\F_1$ are called \emph{Pareto optimal} or \emph{nondominated}. The \emph{second Pareto front}, denoted  $\F_2$, is the set of  minimal elements from $S \setminus \F_1$, and the $k^{\rm th}$ \emph{Pareto front} is defined as
\[\F_k = \text{Minimal elements of } S \setminus \bigcup_{i < k} \F_i.\]
The process of sorting the set $S$ into Pareto fronts, or nondominated layers, is called \emph{nondominated sorting}, and is widely used in multi-objective optimization (see \cite{deb2002,calder2014} and references therein), with recent applications to machine learning~\cite{hsiao2015,hsiao2015b,hsiao2012}. It turns out that nondominated sorting is equivalent to finding the longest chain in a partially ordered set, which has a long history in probability and combinatorics~\cite{ulam1961,hammersley1972,viennot1984,bollobas1988,deuschel1995,felsner1999,prahofer2000}. It was shown in \cite{calder2014} that the Pareto fronts converge almost surely in the limit as $N \to \infty$ to the level sets of the unique nondecreasing\footnote{We say that $u:\Omega\subset \R^n \to \R$ is \emph{nondecreasing} if $x_i \mapsto u(x)$ is nondecreasing for all $i$.}  viscosity solution of \eqref{eq:Pintro}. Figure \ref{fig:demo} gives an illustration of this continuum limit.

\begin{figure}
\centering
\subfigure[\emph{i.i.d.}~samples]{\includegraphics[width=0.32\textwidth,clip=true,trim = 55 30 50 27]{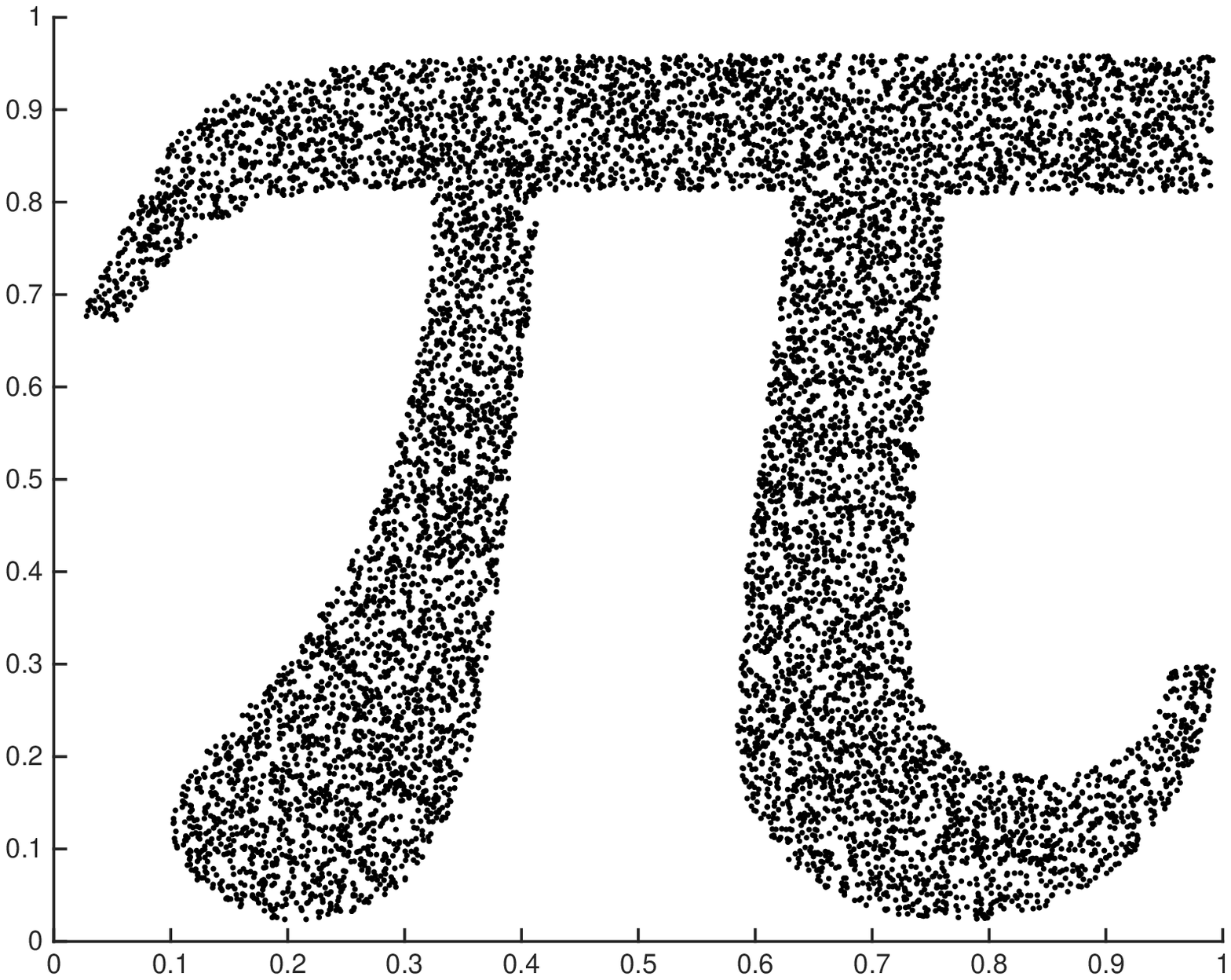}}
\subfigure[$N=10^4$ samples]{\includegraphics[width=0.32\textwidth,clip=true,trim = 55 30 50 27]{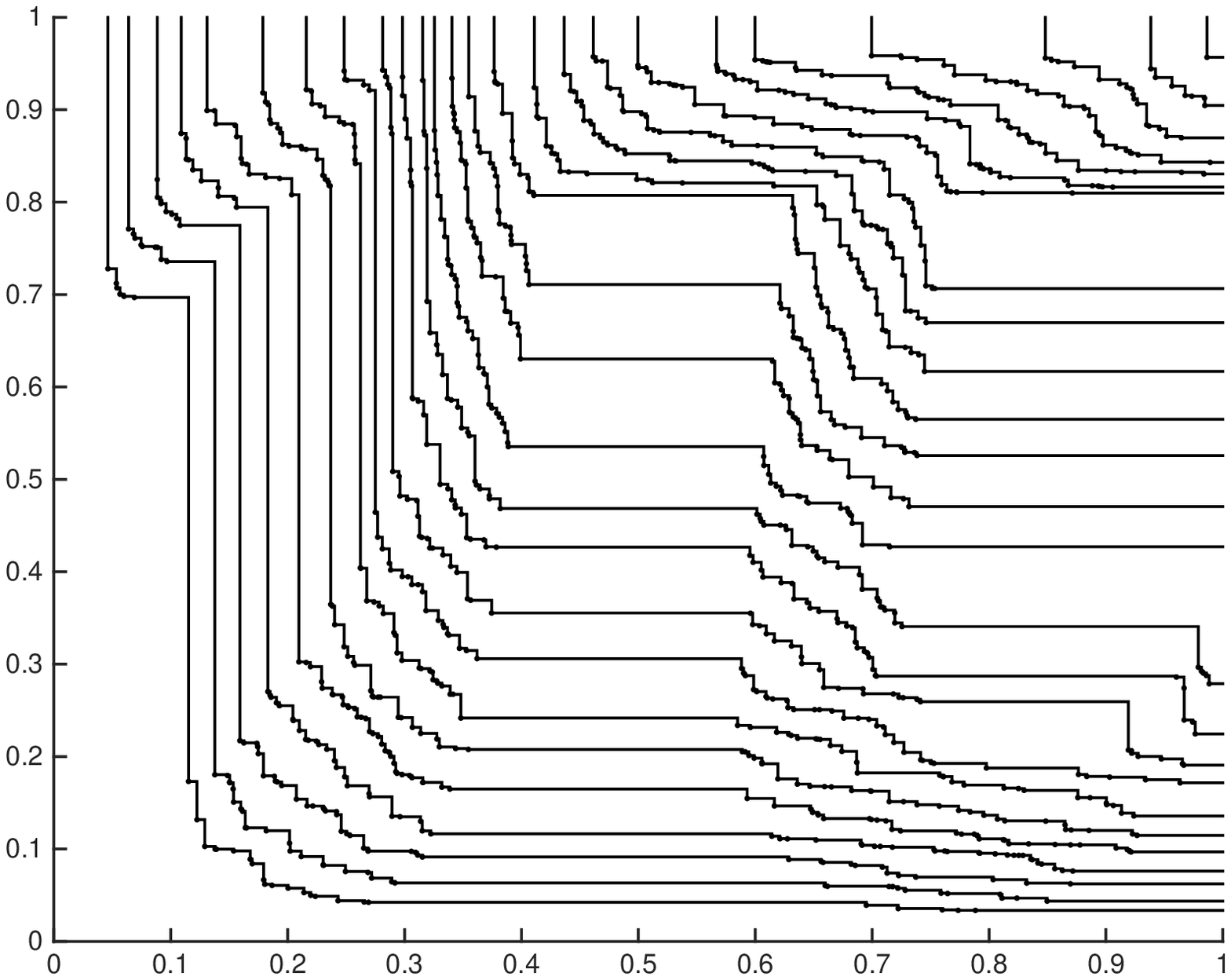}}
\subfigure[$N=10^6$ samples]{\includegraphics[width=0.32\textwidth,clip=true,trim = 55 30 50 27]{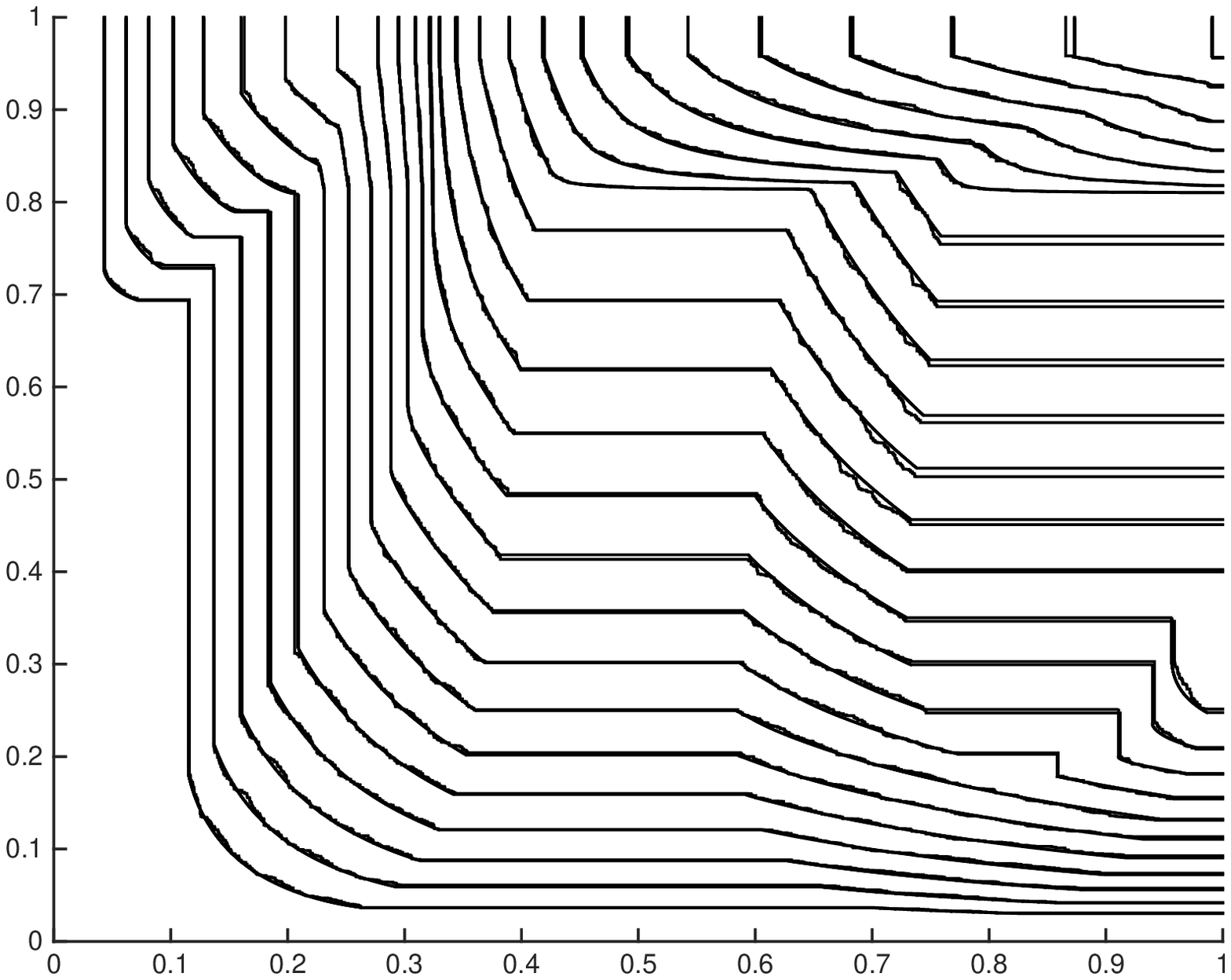}}
\caption{A simulation illustrating that \eqref{eq:Pintro} is the continuum limit of nondominated sorting. The distribution of $X_1,\dots,X_N$ is depicted in (a). In (b) we show 25 of the nondominated layers obtained by sorting $N=10^4$ \emph{i.i.d.}~samples, and in (c) we compare the layers for $N=10^6$ samples against the level sets of the viscosity solution of \eqref{eq:Pintro}.}
\label{fig:demo}
\end{figure}

In \cite{calder2015pde}, a fast algorithm called \emph{PDE-based ranking} was proposed for approximate nondominated sorting of large datasets. The basic idea is to estimate the density $f$ from a (relatively small) subset of $X_1,\dots,X_N$, and then use the numerical solution of the Hamilton-Jacobi equation \eqref{eq:Pintro} as an approximation of nondominated sorting. It was shown in \cite{calder2015pde} that PDE-based ranking is substantially faster than nondominated sorting in relatively low dimensions, while maintaining high levels of sorting accuracy. The numerical scheme for \eqref{eq:Pintro} used in \cite{calder2015pde} is based on backward finite differences and can be solved very efficiently in a single pass. Due to the fact that information flows along coordinate axes in the definition of nondominated sorting, this scheme is upwind (or monotone), and convergence of the scheme was established in \cite{calder2015pde}. 

Although the scheme used in \cite{calder2015pde} is simple and efficient, it suffers from poor accuracy, with formal rates of convergence on the order of $O(h^\frac{1}{n})$ for a grid with spacing $h>0$ in dimension $n$. In this paper, we propose two new and highly efficient finite difference schemes for solving \eqref{eq:Pintro}. Both schemes have a formal accuracy of $O(h)$ when the solution is smooth, and we prove the usual $O(\sqrt{h})$ rates in the context of non-smooth viscosity solutions. These schemes can be used to increase the accuracy of PDE-based ranking~\cite{calder2015pde} \emph{without} increasing computational complexity. Alternatively, with these highly accurate schemes we can afford to use a coarser grid resolution, and thus we can reduce the computational complexity of PDE-based ranking while maintaining high accuracy. This is particularly important in applications of nondominated sorting~\cite{hsiao2015,hsiao2015b,deb2002}, which will benefit from highly accurate and efficient algorithms for sorting massive datasets. We detail the new schemes and our main results in the next section.

\subsection{Main results}
\label{sec:main}

We pose the Hamilton-Jacobi equation \eqref{eq:Pintro} on a compact domain as follows:
\begin{equation}\tag{P1}\label{P1}
\left.
\begin{aligned}
(u_{x_1})_+ \cdots (u_{x_n})_+ &= f& &\text{in } (0,1]^n\\
u &=0& &\text{on } \Gamma, 
\end{aligned}\right\}
\end{equation}
where $\Gamma := \partial [0,1]^n \setminus (0,1]^n$ and $a_+ := \max(a,0)$. Notice we have modified \eqref{eq:Pintro} by taking the positive parts of the partial derivatives. This is necessary to obtain existence of a viscosity solution of \eqref{P1}, and is a well-known issue with viscosity solutions on boundaries of domains (see \cite{crandall1992}). We elaborate on this briefly in Section \ref{sec:comparison}. We should mention that there is no loss of generality in considering the domain $[0,1]^n$ in \eqref{P1}. Indeed, we can make a simple scaling argument to transform the domain of \eqref{P1} into $\prod_{i=1}^n [0,x_i]$ for any $x \in \R^n_+$. 

Let $h>0$, $\Z_h = \{hk \, : \, k \in \Z\}$, and $\Z^n_h = (\Z_h)^n$. For $\Omega \subset \R^n$ set $\Omega_h = \Omega \cap \Z_h^n$. We recall the numerical scheme for \eqref{P1} from \cite{calder2015pde}: 
\begin{equation}\tag{S1}\label{S1}
\left.
\begin{aligned}
(D^-_1u_h(x))_+ \cdots (D^-_nu_h(x))_+ &= f(x)& &\text{if } x \in (0,1]^n_h\\
u_h(x) &=0& &\text{if } x \in \Gamma_h,
\end{aligned}\right\}
\end{equation}
where $u_h:[0,1]^n_h \to \R$ is the numerical solution and 
\[D^\pm_iu_h(x) = \pm\frac{u_h(x\pm he_i)-u_h(x)}{h}.\]
The solution $u_h$ of \eqref{S1} can be solved efficiently in a single pass, which is reminiscent of the fast marching \cite{sethian1996fast} and fast sweeping \cite{zhao2005fast} algorithms. In dimension $n=2$, the scheme is quadratic and can be solved in closed form
\begin{equation}\label{eq:S1closed}
u_h(x) =\frac{u_h(x-he_1) + u_h(x-he_2)}{2} + \frac{1}{2}\sqrt{(u_h(x-he_1) - u_h(x-he_2))^2 + 4h^2f(x)^2 }. 
\end{equation}
In dimensions $n\geq 3$, the scheme can be solved via any iteration method, such as a bisection search.
Convergence of \eqref{S1} to the viscosity solution of \eqref{P1} was established in \cite{calder2015pde}.

While \eqref{S1} is optimal in terms of computational complexity on a fixed grid, its accuracy is at best $O(h^\frac{1}{n})$. To see this, consider the special case of $f\equiv 1$ and $u(x) = n(x_1\cdots x_n)^\frac{1}{n}$. The solution $u$ is smooth on $(0,1]^n$, and has a gradient singularity on the boundary $\Gamma$. We can use the comparison principle \cite{calder2015pde} for \eqref{P1} to show that in general
\[n(x_1\cdots x_n)^\frac{1}{n} \inf_{[0,1]^n} f^\frac{1}{n} \leq u(x) \leq n(x_1\cdots x_n)^\frac{1}{n} \sup_{[0,1]^n} f^\frac{1}{n}.\]
Therefore, the gradient singularity on $\Gamma$ is typical for solutions of \eqref{P1} whenever $\inf_{[0,1]^n} f > 0$. Let $\phi(x) = Cn(x_1\cdots x_n)^\frac{1}{n}$. By the concavity of $\phi$
\[D^-_i\phi(x) \geq \phi_{x_i}(x) = C(x_1\cdots x_n)^\frac{1}{n} x_i^{-1}.\]
On the other hand, if $x_i=h$ then 
\[D^-_i\phi(x) = \frac{\phi(x)}{h} = Cn(x_1\cdots x_n)^\frac{1}{n} x_i^{-1}.\]
Therefore, for any $x \in (0,1]^n_h$ such that $x_i=h$ for some $i$ we have
\[D^-_1\phi(x) \cdots D^-_n\phi(x) \geq nC^n.\]
Setting $C=n^{-\frac{1}{n}}$ and invoking the comparison principle for \eqref{S1} yields
\[u_h(x) \leq \phi(x) =  n^{1-\frac{1}{n}}(x_1\cdots x_n)^\frac{1}{n} \ \ \text{whenever } x_i = h \text{ for some } i.\]
Letting $x=(h,1,\dots,1)$ we find that
\[u_h(x) \leq n^{1-\frac{1}{n}}h^\frac{1}{n} = u(x) - n(1-n^{-\frac{1}{n}}) h^\frac{1}{n}.\]
Therefore, the scheme \eqref{S1} makes an error on the order of $O(h^\frac{1}{n})$ in the immediate vicinity of the boundary $\Gamma$.  Since $u$ is generally not smooth, the best theoretical rate that one can prove in the context of viscosity solutions is typically strictly worse than the formal rate (see \cite{crandall1984two,souganidis1985approximation, deckelnick2004}). In Section \ref{sec:numerics}, we show numerical results indicating that the $\ell^\infty$ convergence rate of $O(h^\frac{1}{n})$ is typically observed in practice.

Since the slow convergence rate is caused by a singularity in the gradient of $u$ on $\Gamma$, it is natural to look for a transformation of $u$ that removes this singularity. With this in mind, we set $v = u^n/n^n$ where $u$ is the nondecreasing viscosity solution of \eqref{P1}. When $f\equiv1$ we have $v(x) = x_1\cdots x_n$, which is Lipschitz continuous (in fact smooth) on $[0,1]^n$. In general, we prove in Lemma \ref{lem:vexun} that $v \in C^{0,1}([0,1]^n)$ whenever $f^\frac{1}{n} \in C^{0,1}([0,1]^n)$. Furthermore, $v$ is a viscosity solution of 
\begin{equation}\tag{P2}\label{P2}
\left.
\begin{aligned}
(v_{x_1})_+ \cdots (v_{x_n})_+ &= v^{n-1}f& &\text{in } (0,1]^n\\
v &=0& &\text{on } \Gamma.
\end{aligned}\right\}
\end{equation}
Since $v$ is Lipschitz continuous, it is reasonable to suspect that a numerical scheme for \eqref{P2} would have a better convergence rate than \eqref{S1}. We therefore propose the following finite difference scheme for \eqref{P2}:
\begin{equation}\tag{S2}\label{S2}
\left.
\begin{aligned}
(D^-_1v_h(x))_+ \cdots (D^-_nv_h(x))_+ &= v_h(x)^{n-1} f(x)& &\text{if } x \in (0,1]^n_h\\
v_h(x) &=0& &\text{if } x \in \Gamma_h,
\end{aligned}\right\}
\end{equation}
where $v_h:[0,1]^n_h \to \R$. We take $v_h(x)$ to be the largest solution of \eqref{S2} at each $x \in (0,1]^n_h$. It is interesting to note that when $f$ is constant, the solution of the scheme \eqref{S2} is $v_h(x) = cx_1\cdots x_n$, which is the \emph{exact} solution of \eqref{P2}. 

The scheme can be solved efficiently in a single pass, similar to \eqref{S1}, and in dimension $n=2$ we have the closed form expression
\begin{equation}\label{eq:S2closed}
v_h(x) =\frac{A + h^2f(x)}{2} + \frac{1}{2}\sqrt{B^2 + 2h^2f(x)A +  h^4f(x)^2 },
\end{equation}
where
\[A = v_h(x-he_1) + v_h(x-he_2) \ \ \text{ and } \ \ B = v_h(x-he_1) - v_h(x-he_2).\]
In Theorem \ref{thm:convPv}, we prove convergence of \eqref{S2} when $f$ is continuous and nonnegative.  Our main result is the following convergence rate.
\begin{theorem}\label{thm:ratev}
Suppose $f \in C^{0,1}([0,1]^n)$ and $f > 0$. Let $v_h$ be the solution of \eqref{S2} and let $u$ be the nondecreasing viscosity solution of \eqref{P1}. Then 
\begin{equation}\label{eq:ratev}
|n^nv_h(x) -u(x)^n| \leq C\sqrt{h} \ \ \text{ for all } x \in [0,1]^n_h,
\end{equation}
and 
\begin{equation}\label{eq:ratev2}
|nv_h(x)^\frac{1}{n} -u(x)| \leq C\delta^{1-n}\sqrt{h} \ \ \text{ for all } x \in [\delta,1]^n_h,
\end{equation}
where $\delta>0$ and $C = C\big(n,[f]_{1;[0,1]^n},\inf_{[0,1]^n}f\big)$.
\end{theorem}
Notice that \eqref{P2} has a zeroth order term with the wrong sign for comparison principle arguments to hold. We can see this by observing that the method of vanishing viscosity takes the form
\[v_{x_1} \cdots v_{x_n} - \eps\Delta v = v^{n-1} f.\]
Since the standard proof of convergence rates for numerical approximations to viscosity solutions is based on the proof of the comparison principle~\cite{crandall1984two,souganidis1985approximation,deckelnick2004}, we cannot directly apply these techniques to \eqref{S2}.

Our proof of Theorem \ref{thm:ratev} passes through an auxiliary problem, which actually suggests another numerical scheme for solving \eqref{P1}. Based on our observation that $u(x) = n(x_1\cdots x_n)^\frac{1}{n}$ is the viscosity solution of \eqref{P1} corresponding to $f\equiv 1$, it is natural in more general settings to make the ansatz 
\begin{equation}\label{eq:uansatz}
u(x) = n(x_1\cdots x_n)^\frac{1}{n} w(x),
\end{equation}
for some function $w:[0,1]^n \to [0,\infty)$. If this ansatz is correct, then $w$ would be a viscosity solution of
\begin{equation}\tag{P3}\label{P3}
\prod_{i=1}^n (w + nx_i w_{x_i})_+ = f\ \ \text{ on } (0,1]^n.
\end{equation}
It turns out that \eqref{P3} is well-posed within the class of bounded viscosity solutions \emph{without} imposing a boundary condition. The boundary condition is actually encoded into the PDE due to the fact that the term $nx_i w_{x_i}$ vanishes on $\Gamma\cap \{x_i=0\}$. This suggests, for example, that we should expect $w(0) = f(0)^\frac{1}{n}$. Due to the degeneracy of the terms $nx_iw_{x_i}$, there are in general infinitely many unbounded viscosity solutions of \eqref{P3}. We characterize $w$ from \eqref{eq:uansatz} as the maximal bounded viscosity solution of \eqref{P3}, and we show in Lemma \ref{lem:wlip} that $w \in C^{0,1}([0,1]^n)$ whenever $f^\frac{1}{n} \in C^{0,1}([0,1]^n)$. 

We propose the following numerical scheme for \eqref{P3}:
\begin{equation}\tag{S3}\label{S3} 
\prod_{i=1}^n \left(w_h(x) + nx_i D^-_iw_h(x)\right)_+= f(x) \ \ \text{for all } x \in [0,1]^n_h,
\end{equation}
where $w_h:[0,1]^n_h \to [0,\infty)$. Here, we take $w_h(x)$ to be the largest solution of \eqref{S3} at each $x \in [0,1]^n_h$. We note that for $x \in [0,1]^n_h$ such that $x_i=0$, the quantity $D^-_iw_h(x)$ is undefined. Since this term appears in the form $nx_iD^-_iw_h(x)$, its value is not used in the scheme \eqref{S3}. It is interesting to note that when $f$ is constant, the scheme \eqref{S3} gives the \emph{exact} solution of \eqref{P3}.

This scheme can be solved efficiently in a single pass, and in dimension $n=2$ the scheme can be solved in closed form
\begin{equation}\label{eq:S3closed}
w_h(x) = C + \sqrt{D^2 + (2x_1 + h)(2x_2+h)h^2f(x)},
\end{equation}
where
\[C = x_1(2x_2 + h)w_h(x-he_1) + x_2(2x_1 + h)w_h(x-he_2),\]
and
\[D = x_1(2x_2 + h)w_h(x-he_1) - x_2(2x_1 + h)w_h(x-he_2).\]

Since the zeroth order term in \eqref{P3} has the correct sign, we can use a modification of the standard convergence proof to establish the following convergence rate.
\begin{theorem}\label{thm:rate1}
Suppose that $f \in C^{0,1}([0,1]^n)$ and $f>0$. Let $w_h$ be the solution of \eqref{S3}, and let $w$ be the maximal bounded viscosity solution of \eqref{P3}. Then 
\begin{equation}\label{eq:rate1}
|w(x) - w_h(x)| \leq C\sqrt{h}\ \ \text{ for all } x \in [0,1]^n_h,
\end{equation}
where $C = C\big(n,[f]_{1;[0,1]^n}, \inf_{[0,1]^n}f\big)$.
\end{theorem}
Our proof of Theorem \ref{thm:ratev} proceeds by first showing (in Lemma \ref{lem:changeofvar}) that 
\[|x_1\dots x_n w_h(x)^n - v_h(x)| \leq Ch,\]
and then invoking Theorem \ref{thm:rate1}.

Although both \eqref{S2} and \eqref{S3} have the same provable convergence rates, our numerical results presented in Section \ref{sec:numerics} suggest that in general \eqref{S2} has a better experimental convergence rate than \eqref{S3}. This can be explained by observing that $u$ can only have gradient singularities when transitioning from zero to a positive value. The transformation $v=u^n/n^n$ regularizes these gradient singularities \emph{anywhere} in the domain $[0,1]^n$, and not just on the boundary $\Gamma$. On the other hand, the transformation $w(x) = n^{-1}(x_1\cdots x_n)^{-\frac{1}{n}}u(x)$ is designed only to capture singularities on the boundary $\Gamma$. In Section \ref{sec:numerics}, we give an example of a discontinuous function $f$ for which \eqref{S2} exhibits a better convergence rate than \eqref{S3} for the reason outlined above.

This paper is organized as follows. In Section \ref{sec:comparison}, we prove comparison principles for viscosity solutions of \eqref{P2} and \eqref{P3}. In Section \ref{sec:conv}, we use the Barles-Souganidis framework~\cite{barles1991} to prove convergence of the schemes \eqref{S2} and \eqref{S3} under the assumption that $f$ is continuous and nonnegative. In Section \ref{sec:rates} we prove Theorems \ref{thm:ratev} and \ref{thm:rate1} establishing rates of convergence for \eqref{S2} and \eqref{S3} when $f$ is positive and Lipschitz. The proofs of the convergence rates require Lipschitz estimates for the viscosity solutions of \eqref{P2} and \eqref{P3}. These are obtained in Section \ref{sec:S3lip}. In Section \ref{sec:numerics}, we show the results of numerical simulations comparing all three schemes.

\section{Some comparison principles}
\label{sec:comparison}

We first prove comparison principles for \eqref{P1}--\eqref{P3} that will be utilized later in the convergence proofs. Let us first briefly comment on the differences between \eqref{P1} and \eqref{eq:Pintro}. As we mentioned in Section \ref{sec:main}, it is necessary to modify \eqref{eq:Pintro} by taking the positive parts of $u_{x_1},\dots,u_{x_n}$ when posing the PDE on compact domains. To see why this is necessary, consider \eqref{P1} with $f\equiv 1$ in dimension $n=2$ without this modification:
\begin{equation}\label{eq:P1n2}
\left.
\begin{aligned}
u_{x_1}u_{x_2} &= 1& &\text{in } (0,1]^2\\
u &=0& &\text{on } \Gamma.
\end{aligned}\right\}
\end{equation}
This Hamilton-Jacobi equation of course has a classical solution $u(x) = 2\sqrt{x_1x_2}$ that is smooth on $(0,1]^2$. However, $u$ is not a viscosity solution of \eqref{eq:P1n2}. To see this, let $\phi(x) = -t(x_1+x_2)$. Then $u - \phi$ has a local maximum at $x=(1,1)$ relative to $(0,1]^2$ for every $t>0$. Since 
\[\phi_{x_1}(1,1)\phi_{x_2}(1,1) = t^2 > f(1,1) \]
for $t>\sqrt{f(1,1)}$, the viscosity subsolution property fails to hold at $x=(1,1)$. This is a well-known issue with viscosity solutions on boundaries of domains (see \cite{crandall1992}). Notice, however, that
\[(\phi_{x_1}(1,1))_+(\phi_{x_2}(1,1))_+ = 0 \leq f(1,1) \ \ \text{ for all } t>0. \]

Taking the positive parts of $u_{x_1},\dots,u_{x_n}$ in \eqref{eq:Pintro} gives the PDE a useful monotonicity property that we will exploit in this paper. Since it is useful to abstract this property, we make the following definition.
\begin{definition}\label{def:directed}
Let $s < 1$. We say $H:(s,1]^n\times \R \times \R^n \to \R$ is \emph{directed} if for all $(x,z) \in (s,1]^n\times\R$
\begin{equation}\label{eq:directed}
p \mapsto H(x,z,p)  \text{ is nondecreasing}.
\end{equation}
\end{definition}

If $H$ is directed, it is simple to construct a monotone (or upwind) numerical scheme using backward difference quotients. Indeed, let us consider the scheme
\begin{equation}\label{eq:S}
H(x,u_h(x),D^-u_h(x)) = 0 \ \ \ \text{in } (s,1]^n_h,
\end{equation}
where $u_h:[s,1]^n_h \to \R$ and 
\[D^-u_h(x) = (D^-_1u_h(x),\dots,D^-_nu_h(x)).\]
To see that \eqref{eq:S} is monotone, suppose that $u(x_0) = v(x_0)$ and $u(x) \geq v(x)$ for $x \in [s,1]^n_h$. Then $D^-_iu(x_0) \leq D^-_iv(x_0)$ for all $i \in \{1,\dots,n\}$, and since $H$ is directed 
\[H(x,u(x),D^-u(x)) \leq H(x,v(x),D^-v(x)).\]
The following Theorem is a direct consequence of this monotonicity and \cite[Theorem 2.1]{barles1991}.
\begin{theorem}\label{thm:bs-conv}
Suppose that $H$ is continuous and directed. Let $\{u_{h}\}_{h>0}$ be solutions of \eqref{eq:S} satisfying
\[\sup_{\substack{h>0\\x \in [s,1]^n_h}} |u_h(x)| < \infty.\] 
Then 
\[\bar{u}(x) = \limsup_{\substack{h\to 0\\y \to x}} u_h(y) \ \ (\text{resp. }\underline{u}(x) = \liminf_{\substack{h\to 0\\y \to x}} u_h(y))\]
is a viscosity subsolution (resp.~supersolution) of $H = 0$ in $(s,1]^n$.
\end{theorem}

When $f \in C([0,\infty)^n)$, there is a unique nondecreasing viscosity solution $u$ of \eqref{eq:Pintro}~\cite{calder2015directed}. We now show that the restriction of $u$ to $[0,1]^n$ is the unique nondecreasing viscosity solution of \eqref{P1}. This establishes the equivalence of \eqref{eq:Pintro} and \eqref{P1}. 
\begin{lemma}\label{lem:comp}
Suppose $f \in C([0,1]^n)$ is nonnegative. Let $u\in \usc([0,1]^n)$ be a viscosity subsolution of \eqref{P1}, and let $v \in \lsc([0,1]^n)$ be a nondecreasing viscosity supersolution of \eqref{P1}.  Then $u\leq v$ on $[0,1]^n$.
\end{lemma}
\begin{proof}
Let $\lambda>0$ and set $v_\lambda = v + \lambda(x_1 + \cdots+ x_n)$. 
Since $v$ is nondecreasing, $v_\lambda$ is a viscosity solution of 
\[(v_{\lambda,x_1})_+ \cdots (v_{\lambda,x_n})_+ \geq  f + \lambda^n \ \ \text{ on } (0,1]^n.\] 
The standard comparison argument based on doubling the variables (see \cite{crandall1992,bardi1997}) shows that  $u \leq v_\lambda$ on $[0,1]^n$. Sending $\lambda \to 0$ completes the proof.
\end{proof}
\begin{lemma}\label{lem:equivalence}
Let $f \in C([0,\infty)^n)$ be nonnegative and let $u$ be the nondecreasing viscosity solution of \eqref{eq:Pintro}. Then the restriction of $u$ to $[0,1]^n$ is the nondecreasing viscosity solution of \eqref{P1}.
\end{lemma}
\begin{proof}
By Lemma \cite[Lemma 3.3]{calder2015pde}, the numerical solutions $u_h$ of \eqref{S1} satisfy the estimate
\[|u_h(x) - u_h(y)| \leq C(|x-y|^\frac{1}{n} + h^\frac{1}{n})\]
for all $x,y \in [0,1]^n_h$, where $C = C(n,\sup_{[0,1]^n}f)$. Combined with Theorem \ref{thm:bs-conv}, this shows that $u_h$ converges uniformly on $[0,1]^n$ to the unique nondecreasing viscosity solution of \eqref{P1}. By \cite[Theorem 3.4]{calder2015pde}, we also have $u_h \to u$ uniformly on $[0,1]^n$, which completes the proof.
\end{proof}

\subsection{The HJ-equation \eqref{P2}}
\label{sec:P2}

In this section, we establish a comparison principle for \eqref{P2}. When $f \in C([0,1]^n)$, the function $v = u^n/n^n$ is a nondecreasing viscosity solution of \eqref{P2} (see Lemma \ref{lem:nonuniq}). We will call $v=u^n/n^n$ the \emph{maximal} viscosity solution of \eqref{P2} (see Lemma \ref{lem:vmaximal}).

Since \eqref{P2} has a zeroth order term of the wrong sign for comparison to hold directly, we find that \eqref{P2} actually has infinitely many nondecreasing viscosity solutions.
\begin{lemma}\label{lem:nonuniq}
Let $f \in C([0,1]^n)$ be nonnegative, let $y \in [0,1]^n$, and let $u$ be the nondecreasing viscosity solution of 
\begin{equation}\label{eq:Pshift}
\left.
\begin{aligned}
(u_{x_1})_+ \cdots (u_{x_n})_+ &= f& &\text{in } \prod_{i=1}^n (y_i,1]\\
u &=0& &\text{on } \prod_{i=1}^n [y_i,1]\setminus \prod_{i=1}^n (y_i,1],
\end{aligned}\right\}
\end{equation}
and extend $u$ to $[0,1]^n$ by setting $u(x) = 0$ for $x\in [0,1]^n \setminus\prod_{i=1}^n [y_i,1]$. Then $v = u^n/n^n$ is a nondecreasing viscosity solution of \eqref{P2}.
\end{lemma}
\begin{proof}
Let $x_0 \in (0,1]^n$ and let $\phi \in C^1([0,1]^n)$ such that $v - \phi$ has a local maximum at $x_0$. We also assume that $\phi(x_0)=v(x_0)$. If $v(x_0) = 0$ then since $u$ is nondecreasing, we see that $v(x) = 0$ for all $x$ that are coordinatewise less than $x_0$. It follows that $\phi_{x_i}(x_0) \leq 0$ for all $i$, and the subsolution property is trivially satisfied. If $v(x_0)>0$, then $u(x_0)>0$ and therefore $x_0 \in \prod_{i=1}^n (y_i,1]$. Setting $\psi(x) = n\phi(x)^\frac{1}{n}$ we find that $u - \psi$ has a local maximum at $x_0$.
Therefore
\[(\psi_{x_1}(x_0))_+ \cdots (\psi_{x_n}(x_0))_+ \leq f(x_0).\]
Since $\phi(x_0)=v(x_0)>0$, this becomes
\[(\phi_{x_1}(x_0))_+ \cdots (\phi_{x_n}(x_0))_+ \leq v(x_0)^{n-1}f(x_0),\]
which verifies the subsolution property. The proof of the supersolution property is similar.
\end{proof}

The lack of uniqueness of nondecreasing viscosity solutions of \eqref{P2} indicates that we cannot expect a comparison principle to hold for arbitrary sub- and supersolutions of \eqref{P2}. However, we show in the following lemma that every subsolution is bounded above by $v=u^n/n^n$. This turns out to be sufficient to prove convergence of \eqref{S2}.
\begin{lemma}\label{lem:vmaximal}
Let $f \in C([0,1]^n)$ be nonnegative and let $v \in \usc([0,1]^n)$ be a nonnegative viscosity subsolution of \eqref{P2}. Then $v \leq u^n/n^n$ on $[0,1]^n$, where $u$ is the nondecreasing viscosity solution of \eqref{P1}.
\end{lemma}
\begin{proof}
We define $u_1 = nv^\frac{1}{n}$. Let $x_0 \in (0,1]^n$ and $\phi \in C^1([0,1]^n)$ such that $u_1 - \phi$ has a local maximum at $x_0$ and $\phi(x_0) = u_1(x_0)$. Letting $\psi(x) = \phi(x)^n/n^n$, we see that $v - \psi$ has a local maximum at $x_0$ and therefore
\[(\psi_{x_1}(x_0))_+ \cdots (\psi_{x_n}(x_0))_+ \leq v(x_0)^{n-1} f(x_0).\]
Since $D\psi(x_0) = v(x_0)^\frac{n-1}{n}D\phi(x_0)$, we have
\[v(x_0)^{n-1}(\phi_{x_1}(x_0))_+ \cdots (\phi_{x_n}(x_0))_+ \leq v(x_0)^{n-1}f(x_0).\]
If $v(x_0)>0$ then 
\begin{equation}\label{eq:suba}
(\phi_{x_1}(x_0))_+ \cdots (\phi_{x_n}(x_0))_+ \leq f(x_0).
\end{equation}
If $v(x_0) = 0$, then since $v$ is nonnegative, $\phi_{x_i}(x_0)\leq 0$ for all $i$, which verifies \eqref{eq:suba}. By Lemma \ref{lem:comp}, $u_1 \leq u$, where $u$ is the unique nondecreasing viscosity solution of \eqref{P1}.
\end{proof}

\subsection{The HJ-equation \eqref{P3}}
\label{sec:P3}

Before establishing a comparison principle for \eqref{P3}, let us comment on the properties of solutions of \eqref{P3}. Let $f \in C([0,1]^n)$ be nonnegative and let $u$ be the nondecreasing viscosity solution of \eqref{P1}. By Lemma \ref{lem:comp} we have
\[0 \leq u(x) \leq n(x_1\cdots x_n)^\frac{1}{n} \sup_{[0,1]^n}f^\frac{1}{n} \ \ \text{ for all } x \in (0,1]^n.\]
Setting $w(x) = n^{-1}(x_1\cdots x_n)^{-\frac{1}{n}}u(x) \in C((0,1]^n)$ we have
\[0 \leq w(x) \leq \sup_{[0,1]^n}f^\frac{1}{n} \ \ \text{ for all } x \in (0,1]^n.\]
We also have that $w$ is a viscosity solution of \eqref{P3} that satisfies 
\begin{equation}\label{eq:wcond}
w + nx_i w_{x_i} \geq 0 \ \ \text{ on } (0,1]^n
\end{equation}
in the viscosity sense for all $i$. To see this, let $x_0 \in (0,1]^n$ and $\phi\in C^1((0,1]^n)$ such that $w-\phi$ has a local minimum at $x_0$. We can also assume $\phi(x_0)=w(x_0)$, so that $w \geq \phi$ in a neighborhood of $x_0$. Setting $\psi(x) = n(x_1\cdots x_n)^\frac{1}{n}\phi(x)$ it follows that $u \geq \psi$ in a neighborhood of $x_0$ and $u(x_0)=\psi(x_0)$. Therefore $u-\psi$ has a local minimum at $x_0$. Since $u$ is nondecreasing, $\phi_{x_i}(x_0)\geq 0$ for all $i$ and 
\[\psi_{x_1}(x_0) \cdots \psi_{x_n}(x_0) \geq f(x_0).\]
A simple computation shows that 
\[\prod_{i=1}^n (w(x_0) + nx_i \phi_{x_i}(x_0)) \geq f(x_0),\]
and 
\[w(x_0) + nx_i \phi_{x_i}(x_0) \geq 0  \ \ \ \text{ for all } i.\]
The subsolution property is verified similarly. We will call $w$ the \emph{maximal bounded} viscosity solution of \eqref{P3}.

Notice the boundary condition $u=0$ on $\Gamma$ is only used to show that $w$ is bounded. Indeed, it is clear that the argument above holds when $u$ is \emph{any} viscosity solution of 
\[(u_{x_1})_+ \cdots (u_{x_n})_+ = f \ \ \ \text{ in } (0,1]^n.\]
This yields an infinite number of unbounded viscosity solutions of \eqref{P3}. For instance, when $f\equiv 1$ the function
\[w(x) = \prod_{i=1}^n (1 + Cx_i^{-1})^\frac{1}{n}\]
for any $C\geq0$ is a viscosity solution of \eqref{P3}. Taking $C=0$ gives the bounded viscosity solution of interest from \eqref{eq:uansatz}. The following theorem characterizes this solution as the unique bounded viscosity solution of \eqref{P3} satisfying \eqref{eq:wcond}.
\begin{theorem}\label{thm:comparison-w}
Assume that $f \in C([0,1]^n)$ is nonnegative. Let $w_1 \in \usc((0,1]^n)$ and $w_2 \in \lsc((0,1]^n)$ be bounded viscosity sub- and supersolutions of \eqref{P3}, respectively, and suppose that $w_2$ satisfies \eqref{eq:wcond} in the viscosity sense for all $i$. Then $w_1\leq w_2$ on $(0,1]^n$.
\end{theorem}
\begin{proof}
For $i=1,2$, we define 
\[u_i(x) = \begin{cases}
 n(x_1\cdots x_n)^\frac{1}{n} w_i(x),& \text{if } x \in (0,1]^n\\
 0,& \text{if } x \in \Gamma.\end{cases}\]
Since $w_1 \in \usc((0,1]^n)$ is bounded, $u_1 \in \usc([0,1]^n)$. Similarly, $u_2 \in \lsc([0,1]^n)$.

We first show that $u_2$ is a nondecreasing viscosity supersolution of \eqref{P1}. Let $y \in (0,1]^n$ and $\phi \in C^1((0,1]^n)$ such that $u_2 - \phi$ has a local minimum at $y$ and $u_2(y) = \phi(y)$.  We define
\[\psi(x) = \frac{\phi(x)}{n(x_1\cdots x_n)^\frac{1}{n}} \ \ \ \text{ for } x \in (0,1]^n.\]
Then it follows that $w_2 - \psi$ has a local minimum at $y$. Since $w_2$ is a viscosity supersolution of \eqref{P3} we have
\begin{equation}\label{eq:super}
\prod_{i=1}^n (w_2(y) + ny_i \psi_{x_i}(y))_+ \geq f(y).
\end{equation}
Since
\[\psi_{x_i}(y) = \frac{\phi_{x_i}(y)}{n(y_1\cdots y_n)^\frac{1}{n}} - \frac{w_2(y)}{ny_i},\]
and $w_2$ satisfies \eqref{eq:wcond} we have
\[0 \leq w_2(y) + ny_i \psi_{x_i}(y) = \frac{y_i \phi_{x_i}(y)}{(y_1\cdots y_n)^\frac{1}{n}}.\]
Therefore $\phi_{x_i}(y) \geq 0$ for all $i$ and 
\[\phi_{x_1}(y) \cdots \phi_{x_n}(y) \geq f(y).\]
This establishes that $u_2$ is a nondecreasing (e.g., see \cite[Lemma 5.17]{bardi1997}) viscosity supersolution of \eqref{P1}.

We can similarly show that $u_1$ is a viscosity subsolution of \eqref{P1}. The proof is completed by invoking Lemma \ref{lem:comp}. 
\end{proof}

\section{Convergence results for continuous $f$}
\label{sec:conv}

In this section we prove convergence of the schemes \eqref{S2} and \eqref{S3} under the assumption that $f$ is continuous and nonnegative. 
 
\subsection{The scheme (S2)}
\label{sec:v}

Let $h>0$ and $x \in (0,1]^n_h$. Given values for  $v_h(x-he_1),\cdots,v_h(x-he_n)$, we define $v_h(x)$ to be the largest solution of \eqref{S2}. If we let $a_i = v(x-he_i)$ for $i=1,\dots,n$ and $b = h^nf(x)$, then this is equivalent to finding the largest solution $t$ of 
\begin{equation}\label{eq:tau}
F(a_1,\dots,a_n,b,t) := \prod_{i=1}^n(t - a_i)_+ - bt^{n-1}=0.
\end{equation} 
 We define
  \[S(a_1,\dots,a_n,b) = \sup\Big\{t \in \R \, : \, F(a_1,\dots,a_n,b,t)=0\Big\}.\]
  Since $t=0$ is always a solution of \eqref{eq:tau}, it is easy to see that $S(x)$ is a nonnegative real number for all $x \in [0,\infty)^{n+1}$. With these definitions, the solution $v_h$ of \eqref{S2} satisfies $v_h(x) = 0$ for $x \in \Gamma$ and 
\begin{equation}\label{eq:maxSv}
v_h(x) = S(v_h(x-he_1),\dots,v_h(x-he_n),h^nf(x)) \ \ \text{ for all } x \in (0,1]^n_h.
\end{equation}
We shall refer to $v_h$ as the \emph{maximal} solution of \eqref{S2}. 

We now establish some important properties of $S$.
\begin{lemma}\label{lem:scheme}
Let $x \in [0,\infty)^{n+1}$.  Then
\begin{enumerate}[label=\emph{(\roman*)}]
\item $S(x) \geq \max\{x_1,\dots,x_n\}$ and $S(x) = \max\{x_1,\dots,x_n\}$ if and only if $x_{n+1}=0$,
\item $F(x,t) > 0$ for all $t > S(x)$, 
\item If $x_{n+1}>0$ then $F(x,t) < 0$ whenever $0 < t < S(x)$, and
\item $S:[0,\infty)^n \to [0,\infty)$ is nondecreasing in all variables.
\end{enumerate}
\end{lemma}
\begin{proof}
By symmetry, we may assume that $x_1\leq x_2\leq \cdots \leq x_n$.

For (i), we simply note that $F(x,x_n) \leq 0$ and $\lim_{t\to \infty} F(x,t) = \infty$. Therefore there exists $t \geq x_n$ such that $F(x,t)=0$. It follows that $S(x)\geq t\geq x_n$. If $x_{n+1} = 0$, then clearly $S(x) = x_n$. Conversely, suppose that $x_{n+1}>0$. If $x_n=0$ then $S(x) = x_{n+1}>0=x_n$. If $x_n > 0$, then $F(x,x_n)<0$ and hence $S(x)>x_n$.

For (ii) If $x_{n+1}=0$, then $S(x) = x_n$, and it is clear that $F(x,t)>0$ for all $t>S(x)$. If $x_{n+1}>0$, then by (i), $S(x)>x_n$ and $F(x,S(x))=0$.  For any $t>x_n$ such that $F(x,t)=0$ we have
\begin{align}\label{eq:Ftpos}
F_t(x,t) &= \sum_{j=1}^n \prod_{i\neq j} (t-x_i) - (n-1)x_{n+1}t^{n-2}\notag\\
&= x_{n+1}t^{n-1}\sum_{j=1}^n \frac{1}{t-x_j} - (n-1)x_{n+1}t^{n-2}\notag\\
&\geq n x_{n+1}t^{n-2} - (n-1)x_{n+1}t^{n-2}= x_{n+1}t^{n-2}>0.
\end{align}
It follows that $F(x,t) > 0$ for all $t>S(x)$.  This establishes (ii).

For (iii), suppose first that $x_n=0$. Then $F(x,t) = t^n - x_{n+1}t^{n-1}$ and $S(x) = x_{n+1}$, from which (iii) immediately follows. If $x_n>0$, then since $x_{n+1}>0$ we have $F(x,t)<0$ for all $0 < t \leq x_n$. Therefore there exists $\eps>0$ such that $F(x,t)<0$ for $0 < t < x_n+ \eps$. Define
\[\tau = \sup \Big\{ t \in \R \, : \, F(x,s) < 0 \text{ for all } s \in (0,t)\Big\}.\]
Clearly $F(x,\tau)=0$ and $\tau > x_n+\eps$. For any $t\geq\tau$ satisfying $F(x,t)=0$, we have by \eqref{eq:Ftpos} that $F_t(x,t)>0$. It follows that $\tau = S(x)$, which establishes (iii).

For (iv) we set
\[U = \big\{x \in [0,\infty)^{n+1} \, : \, \max\{x_1,\dots,x_n\}>0 \ \text{ and } \ x_{n+1}>0\big\}.\] 
By (i), $S(x)>x_n$ for every $x \in U$. Therefore $F_t(x,S(x))>0$ for all $x \in U$, and it follows from the implicit function theorem that the restriction of $S$ to $U$ is smooth. 

Since $F(x,S(x))=0$ we have
\begin{equation}\label{eq:defS}
\sum_{j=1}^n\log(S(x) - x_j) =\log(x_{n+1}) + (n-1)\log(S(x)).
\end{equation}
Differentiating \eqref{eq:defS} in $x_i$ for $i\in \{1,\dots,n\}$ we obtain
\[\left(\sum_{j=1}^n \frac{1}{S(x) - x_j} - \frac{n-1}{S(x)}\right)S_{x_i}(x) = \frac{1}{S(x)-x_i}.\]
Since $S(x)>x_j$ for all $j\in \{1,\dots,n\}$, we have
\[\sum_{j=1}^n \frac{1}{S(x) - x_j} \geq \sum_{j=1}^n \frac{1}{S(x)}= \frac{n}{S(x)}.\]
It follows that $S_{x_i}(x)>0$. Differentiating \eqref{eq:defS} in $x_{n+1}$ we have
\[\left(\sum_{j=1}^n \frac{1}{S(x) - x_j} - \frac{n-1}{S(x)}\right)S_{x_{n+1}}(x) = \frac{1}{x_{n+1}}.\]
As before, it follows that $S_{x_{n+1}}(x)>0$. Therefore $S$ is strictly increasing on $U$. 

If $x_{n+1}=0$, then $S(x) = \max\{x_1,\dots,x_n\}$ is nondecreasing. If $x_n = 0$, then $S(x) = x_{n+1}$ is again nondecreasing. The continuity of $S$ establishes (iv).
\end{proof}
\begin{remark}\label{rem:lipS}
In the proof of Lemma \ref{lem:scheme} (iv), we can use the inequality 
\[\sum_{j=1}^n \frac{1}{S(x) - x_j} \geq \frac{1}{S(x) - x_i} + \sum_{j\neq i} \frac{1}{S(x)}= \frac{1}{S(x) - x_i}+  \frac{n-1}{S(x)}\]
to find that $S_{x_i}(x) \leq 1$ for all $x \in U$ and $i\in \{1,\dots,n\}$. Since $S(0,\dots,0,x_{n+1}) = x_{n+1}$, we have the bound
\begin{equation}\label{eq:boundS}
S(x) \leq \sum_{i=1}^{n+1} x_i.
\end{equation}
\end{remark}
Using the properties of $S$ from Lemma \ref{lem:scheme} we can establish a comparison principle for the scheme \eqref{S2}.
\begin{theorem}\label{thm:Svcomp}
Let $h>0$ and suppose $f\geq 0$. Let $v_1$ be a subsolution of \eqref{S2} and let $v_2$ be a supersolution of \eqref{S2} satisfying 
\begin{equation}\label{eq:Ssuper}
v_2(x) \geq S(v_2(x-he_1),\dots,v_2(x-he_n),h^nf(x)) \ \ \text{ for all } x \in (0,1]^n_h.
\end{equation}
Then $v_1 \leq v_2$ on $[0,1]^n_h$.
\end{theorem}
\begin{proof}
We will prove the result by induction. We have $v_1(x) \leq v_2(x)$ for $x \in \Gamma_h$ by definition. Now let $x \in (0,1]^n_h$ and suppose that
\begin{equation}\label{eq:induct}
v_1(x-he_i) \leq v_2(x-he_i) \ \ \ \text{ for } i=1,\dots,n.
\end{equation}

Since $v_1$ is a subsolution of \eqref{S2} we have
\[F(v_1(x-he_1),\dots,v_1(x-he_n),h^nf(x),v_1(x))\leq 0.\]
It follows from Lemma \ref{lem:scheme} (ii) that 
\begin{equation}\label{eq:subsol}
v_1(x) \leq S(v_1(x-he_1),\dots,v_1(x-he_n),h^nf(x)).
\end{equation}
Recalling \eqref{eq:Ssuper} and Lemma \ref{lem:scheme} (iv)  we have
\begin{align*}
v_1(x) &\leq S(v_1(x-he_1),\dots,v_1(x-he_n), h^nf(x)) \\
&\leq S(v_2(x-he_1),\dots,v_2(x-he_n), h^nf(x)) \leq v_2(x).
\end{align*}
The proof is completed by induction.
\end{proof}
\begin{remark}\label{rem:poscomp}
If $f$ is positive and $v$ is a supersolution of \eqref{S2} that is positive on $(0,1]^n_h$, then it follows from Lemma \ref{lem:scheme} (iii) that 
\begin{equation}
v(x) \geq S(v(x-he_1),\dots,v(x-he_n),h^nf(x)) \ \ \text{ for all } x \in (0,1]^n_h.
\end{equation}
\end{remark}
\begin{remark}\label{rem:S2exact}
Notice that 
\[\bar{v}(x) = (x_1\cdots x_n) \sup_{[0,1]^n}f \ \ \text{ and } \ \  \underline{v}(x) =(x_1\cdots x_n) \inf_{[0,1]^n}f\] 
are super- and subsolutions of \eqref{S2}, respectively. By Theorem \ref{thm:Svcomp} and Remark \ref{rem:poscomp}
\begin{equation}\label{S2bound}
(x_1\cdots x_n) \inf_{[0,1]^n} f \leq v_h(x) \leq (x_1\cdots x_n) \sup_{[0,1]^n}f \ \ \text{ for all } x\in [0,1]_h^n.
\end{equation}
As a consequence, if $f\equiv C\geq 0$ then $v_h(x) = Cx_1\cdots x_n$, which is \emph{exactly} equal to the maximal viscosity solution of \eqref{P2}.
\end{remark}
We can now establish convergence of the scheme \eqref{S2}.
\begin{theorem}\label{thm:convPv}
Let $f \in C([0,1]^n)$ be nonnegative, and for each $h>0$ let $v_h$ denote the maximal solution of \eqref{S2}. Then $v_h \to v$ uniformly on $[0,1]^n$ as $h \to 0$, where $v$ is the maximal viscosity solution of \eqref{P2}.
\end{theorem}
\begin{proof}
Let 
\[\bar{v}(x) = \limsup_{\substack{h\to 0 \\ y \to x}} v_h(y) \ \ \text{ and } \ \ \underline{v}(x) = \liminf_{\substack{h\to 0 \\ y \to x}} v_h(y).\]
By Theorem \ref{thm:bs-conv} and Remark \ref{rem:S2exact}, $\bar{v} \in \usc([0,1]^n)$ is a viscosity subsolution of \eqref{P2}. By Lemma \ref{lem:vmaximal}, $\bar{v} \leq v$.

Let $u_h$ be the solution of \eqref{S1}, and let $\psi_h(x) = u_h(x)^n/n^n$. Since $t \mapsto t^n$ is convex for $t>0$ and $u_h$ is nondecreasing, we have
\[D^-_i\psi_h(x) = \frac{u_h(x)^n - u_h(x-he_i)^n}{n^nh}\leq \frac{u_h(x)^{n-1}}{n^{n-1}} D^-_iu_h(x).\]
Therefore
\[D^-_1\psi_h(x) \cdots D^-_n\psi_h(x) \leq \psi_h(x)^{n-1} f(x) \ \ \text{for all } x \in (0,1]^n_h.\]
By Theorem \ref{thm:Svcomp}, $\psi_h \leq v_h$ on $[0,1]^n_h$. As in the proof of Lemma \ref{lem:equivalence}, we have that $\psi_h \to v$ uniformly on $[0,1]^n$ as $h\to 0$. It follows that $\underline{v} \geq v$, which completes the proof.
\end{proof}
\begin{remark}\label{rem:onesided}
Notice in the proof of Theorem \ref{thm:convPv} we showed that $u_h^n \leq n^n v_h$, where $u_h$ is the numerical solution of \eqref{S1}, and $v_h$ is the solution of \eqref{S2}. Under the assumptions of Theorem \ref{thm:ratev}, this gives a one-sided convergence rate for \eqref{S1} of the form
\begin{equation}\label{eq:onesideed}
u_h^n - u^n \leq C\sqrt{h},
\end{equation}
where $u$ is the nondecreasing viscosity solution of \eqref{P1}. When $u$ is concave, we can actually prove that $u_h \leq u$ (see \cite[Lemma 3.5]{calder2015pde}).
\end{remark}

\subsection{The scheme (S3)}
\label{sec:second}

Recall that we defined the solution $w_h:[0,1]^n_h\to \R$ of \eqref{S3} inductively by taking $w_h(x)$ to be the largest solution of the polynomial equation defining \eqref{S3} at each $x$. It is easy to see that 
\begin{equation}\label{eq:key}
w_h(x) + nx_iD^-_iw_h(x) \geq 0 \ \ \text{ for all } x \in [0,1]^n_h \  \text{ and } \ i \in \{1,\dots,n\}.
\end{equation}
The following lemma shows that \eqref{S3} admits a comparison principle whenever the supersolution satisfies \eqref{eq:key}.
\begin{lemma}\label{lem:S3comp}
Suppose $f$ is  nonnegative. Let $w_1$ and $w_2$ be sub- and supersolutions of \eqref{S3}, respectively, and suppose $w_2$ satisfies \eqref{eq:key}. Then $w_1 \leq w_2$ on $[0,1]^n_h$.
\end{lemma}
\begin{proof}
Let $x \in [0,1]^n_h$ be a maximum of $w_1 - w_2$, and suppose to the contrary that $w_1(x) > w_2(x)$. Then we have
\[D^-_iw_1(x) \geq D^-_iw_2(x) \ \ \text{for all } i \text{ such that } x_i\neq 0.\]
Recalling \eqref{eq:key} we have
\[f(x) \leq \prod_{i=1}^n (w_2(x) + nx_iD^-_iw_2(x)) <\prod_{i=1}^n (w_1(x) + nx_iD^-_iw_1(x)) \leq f(x),\]
which is a contradiction.
\end{proof}
We now prove convergence of \eqref{S3}.
\begin{theorem}\label{thm:convS3}
Let $f \in C([0,1]^n)$ be nonnegative, and for $h>0$ let $w_h$ be the maximal solution of \eqref{S3}.  Then $w_h \to w$ uniformly on $(0,1]^n$ as $h \to 0$, where $w$ is the maximal bounded viscosity solution of \eqref{P3}.
\end{theorem}
\begin{proof}
Let
\[\bar{w}(x) = \limsup_{\substack{h\to 0 \\ y \to x}} w_h(y) \ \ \text{ and } \ \ \underline{w}(x) = \liminf_{\substack{h\to 0 \\ y \to x}} w_h(y).\]
We can use the comparison principle from Lemma \ref{lem:S3comp} to show that
\[\inf_{[0,1]^n}f^\frac{1}{n} \leq w_h  \leq \sup_{[0,1]^n}f^\frac{1}{n}.\]
Therefore
\[\inf_{[0,1]^n}f^\frac{1}{n} \leq \underline{w} \leq \bar{w} \leq \sup_{[0,1]^n}f^\frac{1}{n}.\]
By Theorem \ref{thm:bs-conv}, $\bar{w}\in \usc((0,1]^n)$ is a bounded viscosity subsolution of \eqref{P3} and $\underline{w}\in \lsc((0,1]^n)$ is a bounded viscosity supersolution of \eqref{P3}. Furthermore, since $w_h$ satisfies \eqref{eq:key} another application of Theorem \ref{thm:bs-conv} shows that  $\underline{w}$ is a viscosity solution of \eqref{eq:wcond} for all $i$. By Theorem \ref{thm:comparison-w}, $\bar{w}=\underline{w}=w$, where $w$ is the maximal bounded viscosity solution of \eqref{P3}. 
\end{proof}

\section{Rates of convergence for positive and Lipschitz $f$}
\label{sec:rates}

In order to obtain a rate of convergence, it is necessary to show that the either the numerical solution or the viscosity solution of the continuum equation is Lipschitz continuous up to the boundary $\Gamma$. Since none of the Hamiltonians considered in this paper are coercive, the standard textbook estimates~\cite{bardi1997} do not apply.

\subsection{A Lipschitz estimate in dimension $n=2$}
\label{sec:lipv}

In dimension $n=2$, we can prove the required Lipschitz estimate for $v$ directly by differentiating \eqref{P2} and using a comparison principle. Let us briefly sketch the argument, which is made rigorous at the level of the numerical scheme \eqref{S2} in Theorem \ref{thm:vlipd2}. 

Formally differentiating \eqref{P2} in the variable $x_1$ we have
\[v_{x_1 x_1}v_{x_2} + v_{x_1}v_{x_1 x_2} = vf_{x_1} + fv_{x_1}.\]
Setting $\psi = v_{x_1}$ and noting that $v(x) \leq \sup_{[0,1]^2} f$ we have
\begin{equation}\label{eq:diffpde}
\psi_{x_1}v_{x_2} + (\psi_{x_2} - f)\psi \leq [f]_{1;[0,1]^2}\sup_{[0,1]^2} f.
\end{equation}
Consider using a comparison function of the form $\bar{\psi}(x) = C(1+x_2)$. Since $\bar{\psi}_{x_1}\equiv 0$, the first term in \eqref{eq:diffpde} can be ignored. Furthermore, for $C>\sup_{[0,1]^2}f$, the sign of the zeroth order term in \eqref{eq:diffpde} is positive, which suggests that a comparison principle should hold, allowing us to show that $v_{x_1}=\psi \leq C(1+x_2)$ for a possibly larger constant $C>0$. 

Of course, $v$ is not in general smooth enough to use this argument directly, so we need to regularize \eqref{P2}. One obvious approach would be the method of vanishing viscosity. However, there are some technical challenges with this approach, due to the sign of the zeroth order term in \eqref{P2} and the (non-obvious) fact that for positive viscosity, $v^\eps$ can fail to be nondecreasing. Our approach is to apply the argument above at the level of the numerical scheme \eqref{S2}. Since we showed convergence of \eqref{S2} in Theorem \ref{thm:convPv}, a Lipschitz estimate on $v_h$ directly carries over to the maximal viscosity solution $v$ of \eqref{P2}. 

Heading in this direction, we now recall some basic properties of finite differences.
\begin{proposition}\label{prop:finite-diff-properties}
For all $k,j \in \{1,\dots,n\}$ we have
\begin{enumerate}[label=\emph{(\roman*)}]
\item $\displaystyle D^\pm_k D^\pm_j u = D^\pm_j D^\pm_k u$ and $\displaystyle D^+_k D^-_j u = D^-_j D^+_k u$,
\item $D_k^\pm(uv)(x) = u(x) D_k^\pm v(x) + v(x\pm he_k)D_k^\pm u(x)$,
\item $D_k^\pm(uv) = u D_k^\pm v + vD_k^\pm u \pm hD_k^\pm u D_k^\pm v$,
\end{enumerate}
\end{proposition}
 We now give a preliminary regularity result for $v_h$.
\begin{theorem}\label{thm:vlipd2}
Let $n=2$, let $f \in C^{0,1}([0,1]^2)$ be nonnegative, and let $v_h$ be any nondecreasing solution of \eqref{S2}.   Then for $k=1,2$ we have
\begin{equation}\label{eq:vlipd2}
0 \leq D^-_kv_h(x) \leq [f]_{1;[0,1]^2} + 3\sup_{[0,1]^2}f \ \ \text{ for all } x \in (0,1]^2_h.
\end{equation}
\end{theorem}
\begin{proof}
Let us prove the result for $k=1$; the case of $k=2$ follows by symmetry. We extend $v_h$ to a function on $[-h,1]^2_h$ by setting $v_h(x) = 0$ for $x \in [-h,1]^2_h\setminus [0,1]^2_h$. With these definitions, we have $v_h(x) = D^-_iv_h(x)=0$ for $x \in \Gamma_h$. Since $v_h$ is nondecreasing, we therefore have
\[D^-_1v_h(x) D^-_2v_h(x) = v_h(x) f(x) \ \ \text{for all } x \in [0,1]^2_h.\]
We now apply the operator $D^-_1$ to the equation above and use Proposition \ref{prop:finite-diff-properties} (ii) to obtain
\[ D^-_1v_h(x) D^-_1D^-_2v_h(x) + D^-_1D^-_1v_h(x) D^-_2v_h(x-he_1) = v_h(x) D^-_1f(x) + f(x-he_i)D^-_1v_h(x) \]
for all $x \in (0,1]^2_h$. Setting $\psi(x) = D^-_1v_h(x)$ and using Proposition \ref{prop:finite-diff-properties} (i) we have
\[ \psi(x) D^-_2\psi(x) +  D^-_2v_h(x-he_1)D^-_1\psi(x) = v_h(x) D^-_1f(x) + f(x-he_i)\psi(x).\]
Applying Theorem \ref{thm:Svcomp} and Remark \ref{rem:S2exact}
\begin{equation}\label{eq:linearization}
D^-_2v_h(x-he_1)D^-_1\psi(x) + (D^-_2\psi(x)  - f(x-he_i)) \psi(x) \leq [f]_{1;[0,1]^2}\sup_{[0,1]^2}f=:M
\end{equation}
for all $x \in (0,1]^2_h$, and $\psi(x) = 0$ for $x \in\Gamma_h$. We can assume $M>0$, otherwise $v_h(x) = cx_1x_2$ and \eqref{eq:vlipd2} is trivial. 
  
Let $\bar{\psi}(x) = \left(\sqrt{M} + \sup_{[0,1]^2}f\right)(1 + x_2)$ and note that since $v_h$ is nondecreasing we have
\begin{equation}\label{eq:linearization-super}
D^-_2v_h(x-he_1)D^-_1\bar{\psi}(x) + (D^-_2\bar{\psi}(x)  - f(x-he_i)) \bar{\psi}(x) \geq M + \sqrt{M}\sup_{[0,1]^2}f.
\end{equation}
We claim that $\psi \leq \bar{\psi}$. Assume by way of contradiction that $\max_{[0,1]^2_h}(\psi - \bar{\psi}) > 0$. Since $\psi = 0 \leq \bar{\psi}$ on $\Gamma_h$, $\phi-\bar{\phi}$ must attain its positive maximum at some $x \in (0,1]^2_h$.
It follows that $D^-_i\psi(x) \geq D^-_i\bar{\psi}(x)$ for $i=1,2$ and $\psi(x) > \bar{\psi}(x)$. Since $D^-_2\bar{\psi}(x)  - f(x-he_i) > 0$ and $v_h$ is nondecreasing, we can combine \eqref{eq:linearization} and \eqref{eq:linearization-super} to find that 
\[M+ \sqrt{M}\sup_{[0,1]^2}f \leq M,\]
which is a contradiction to the positivity of $M$. By Cauchy's inequality
\[D^-_iv_h(x) = \psi(x)\leq \left(\sqrt{M} + \sup_{[0,1]^2}f\right)(1 + x_2) \leq [f]_{1;[0,1]^2} + 3\sup_{[0,1]^2}f,\]
for all $x \in (0,1]^2_h$.
\end{proof}
Since $v_h \to v$ uniformly (Theorem \ref{thm:convPv}), we can extend the Lipschitz estimate in Theorem \ref{thm:vlipd2} to \eqref{P2}.
\begin{corollary}\label{cor:vlipd2}
Let $n=2$, let $f \in C^{0,1}([0,1]^2)$ be nonnegative, and let $v$ be the maximal viscosity solution of \eqref{P2}. Then there exists $C>0$ such that 
\begin{equation}\label{eq:vlipd2cont}
[v]_{1;[0,1]^2} \leq C \|f\|_{C^{0,1}([0,1]^n)}.
\end{equation}
\end{corollary}
The proof of Theorem \ref{thm:vlipd2} does not work in dimensions $n\geq 3$ due to the additional nonlinearity of \eqref{P2}, which increases with $n$. In Section \ref{sec:second} we study the auxiliary problem \eqref{S3}, which allows us to prove Lipschitz regularity of the maximal viscosity solution of \eqref{P2} in arbitrary dimension (see Lemma \ref{lem:vexun}), provided the strictly stronger condition $f^\frac{1}{n} \in C^{0,1}([0,1]^n)$ holds. We expect Theorem \ref{thm:vlipd2} to hold in dimensions $n\geq 3$ under the natural condition $f \in C^{0,1}([0,1]^n)$, but we currently do not know how to prove this. 

\subsection{The scheme (S3)}
\label{sec:ratesS3}

In this section we prove Lipschitz regularity and rates of convergence for the scheme \eqref{S3}. This establishes one of our main results (Theorem \ref{thm:rate1}). These results are used in Section \ref{sec:ratesS2} to prove similar rates of convergence for \eqref{S2}.

Throughout this section we set
\[H(x,z,p) = \prod_{i=1}^n (z + nx_i p_i)_+,\]
for $(x,z,p) \in \R^n \times \R \times \R^n$.

\subsubsection{Lipschitz regularity}
\label{sec:S3lip}

Let us first give a formal argument suggesting a Lipschitz estimate on the maximal bounded viscosity solution $w$ of \eqref{P3}. Assuming $w$ is smooth and $f>0$, we can take the logarithm of \eqref{P3} to obtain
\[\sum_{j=1}^n \log(w + nx_j w_{x_j}) = \log(f)\]
Differentiating both sides in the variable $x_i$ and setting $\phi = w_{x_i}$ we have
\[\sum_{j=1}^n \frac{(1 + \delta_{i,j})\phi + nx_j\phi_{x_j}}{w + nx_j w_{x_j}} = \frac{f_{x_i}}{f},\]
where $\delta_{i,j}=1$ if $i=j$ and $\delta_{i,j} = 0$ if $i\neq j$.
Notice that at a positive maximum of $\phi$ we have $\phi_{x_j}=0$ for all $j$ and hence
\[nf^{-\frac{1}{n}}\phi = n\phi\prod_{j=1}^n(w + nx_j w_{x_j})^{-\frac{1}{n}}\leq \sum_{j=1}^n \frac{(1 + \delta_{i,j})\phi}{w + nx_j w_{x_j}} = \frac{f_{x_i}}{f},\]
where we employed the inequality of arithmetic and geometric means in the first inequality. This suggests that  
\begin{equation}\label{eq:wx}
w_{x_i} \leq \sup_{[0,1]^n} \frac{1}{n} f^{1-\frac{1}{n}}f_{x_i} = \sup_{[0,1]^n} \partial_{x_i}( f^\frac{1}{n}).
\end{equation}

In Theorem \ref{thm:lip} and Lemma \ref{lem:wlip}, we make this argument rigorous by applying it to the numerical scheme \eqref{S3}. One obvious gap in the argument above is the existence of a positive maximum for $w_{x_i}$. Indeed, when $w$ is any of the infinitely many unbounded viscosity solutions of \eqref{P3}, $w_{x_i}$ either has no maximum value, or no minimum value, on $(0,1]^n$. This is a reflection of the fact that \eqref{S3} has no boundary condition. This issue is resolved by showing (in Proposition \ref{prop:extension}) that the numerical solution $w_h$ can be extended by projection to a solution of \eqref{S3} on the domain $(-\infty,1]^n_h$. 

For $x \in \R^n$, let us set 
\[x_+ =  ((x_1)_+,\dots,(x_n)_+).\]
\begin{proposition}\label{prop:extension}
Suppose $f$ is nonnegative, and let $w$ be a sub- (resp.~super-) solution of \eqref{S3}. Then $\bar{w}(x) := w(x_+)$ is a sub- (resp.~super-) solution of
\begin{equation}\label{eq:extension}
H(x_+,\bar{w}(x),D^-\bar{w}(x)) = \bar{f}(x) \ \ \text{for all } x \in (-\infty,1]^n_h,
\end{equation}
where $\bar{f}(x):=f(x_+)$.
\end{proposition}
\begin{proof}
We claim that 
\[ (x_i)_+ D^-_i\bar{w}(x) = (x_i)_+ D^-_iw(x_+) \ \ \text{ for all } x \in (-\infty,1]^n_h, \, i \in \{1,\dots,n\}.\]
If $x_i \leq 0$ the result is trivial. Hence we may assume that $x_i \geq h$, so that $(x-he_i)_+ = x_+ - he_i$. Then
\[ D^-_i\bar{w}(x) =  \frac{w(x_+) - w((x-he_i)_+)}{h} = \frac{w(x_+) - w(x_+-he_i)}{h} = D^-_iw(x_+),\]
which establishes the claim.

For $x \in (-\infty,1]^n_h\setminus [0,1]^n_h$,  we have $x_+ \in [0,1]^n_h$ and 
\[ \prod_{i=1}^n \left(\bar{w}(x) + n(x_i)_+ D^-_i\bar{w}(x)\right)_+= \prod_{i=1}^n \left(w(x_+) + n(x_i)_+ D^-_iw(x_+)\right)_+.\]
Since $\bar{f}(x) = f(x_+)$, the result immediately follows.
\end{proof}
We also note that we can extend the product rule for finite differences in Proposition \ref{prop:finite-diff-properties} (iii) by induction as follows:
\begin{lemma}\label{lem:prod-rule}
For all $k \in \{1,\dots,n\}$ and $N\geq 2$ we have
\begin{equation}\label{eq:prod-rule}
D^\pm_k(u_1\cdots u_N) = \sum_{j=1}^N D^\pm_k u_j \prod_{i\neq j} u_i+ \sum_{j=2}^N (\pm h)^{j-1} \sum_{|I|=j} \prod_{i \in I} D_k^\pm u_i \prod_{i\not\in I} u_i, 
\end{equation}
where the final summation is over all $I \subset \{1,\dots,N\}$ with $|I|=j$.
\end{lemma}
The proof of Lemma \ref{lem:prod-rule} is deferred to the appendix. We now establish a Lipschitz estimate on the solution $w_h$ of \eqref{S2}.
\begin{theorem}\label{thm:lip}
Suppose that $f \in C^{0,1}([0,1]^n)$ is positive on $[0,1]^n$, and let $w_h$ be the maximal solution of \eqref{S3}. Then 
\begin{equation}\label{eq:bound}
\inf_{x \in U_k} \frac{1}{n} f(x + he_k)^{\frac{1}{n}-1} (D^+_kf(x))_- \leq D^+_kw_h(x) \leq \sup_{U_k} \frac{1}{n} f^{\frac{1}{n}-1} (D^+_kf)_+.
\end{equation}
for all $x \in U_k := \{x \in [0,1]^n_h \, : \, x_k \leq 1-h\}$. In particular
\begin{equation}\label{eq:alt_bound}
\sup_{U_k}|D^+_kw_h| \leq \frac{1}{n}[f]_{1;[0,1]^n}\sup_{[0,1]^n}f^{\frac{1}{n}-1}.
\end{equation}
\end{theorem}
\begin{proof}
For notational convenience, we will write $w$ in place of $w_h$ throughout the proof.

By Proposition \eqref{eq:extension}, we can extend $w$ and $f$ to functions on $(-\infty,1]^n_h$ by setting $w(x) = w(x_+)$ and $f(x) = f(x_+)$, and we have
\begin{equation}\label{eq:what}
H(x_+,w(x),D^-w(x)) = f(x) \ \ \text{for all } x \in (-\infty,1]^n_h.
\end{equation}
Let $\phi(x) = D^+_kw(x)$ for all $x \in (-\infty,1]^n_h$ with $x_k \leq 1-h$. 

The proof is split into two steps.

1. We first show that
\begin{equation}\label{eq:b1}
\sup_{U_k} D^+_kw \leq \sup_{U_k} \frac{1}{n} f^{\frac{1}{n}-1} (D^+_kf)_+.
\end{equation}
We may assume $\phi$ is positive somewhere, otherwise \eqref{eq:b1} trivially holds. For any $x$ such that $x_k \leq -h$, we have $x_+ = (x+he_k)_+$. This implies that $\phi(x) = 0$. Likewise, whenever $x_k \geq 0$, $(x+he_k)_+ = x_+ + he_k$ and we have $\phi(x) = \phi(x_+)$. It follows that $\phi$ attains its positive maximum value at some $x_0\in U_k$. Therefore
\begin{equation}\label{eq:z}
\phi(x_0) > 0 \ \text{ and } \  D^-_j \phi(x_0) \geq 0  \ \text{ for all } j \in \{1,\dots,n\}.
\end{equation}   

Since $f(x) > 0$,
\begin{equation}\label{eq:scheme2}
f(x) = H(x_+,w(x),D^-w(x))= \prod_{j=1}^n a_j(x),
\end{equation}
where
\[a_j(x) = w(x) + n(x_j)_+ D^-_jw(x) > 0.\] 
By Lemma \ref{lem:prod-rule}
\begin{equation}\label{eq:dk}
D^+_kf(x) = \sum_{j=1}^n D^+_k a_j(x)\prod_{i\neq j} a_i(x) + \sum_{j=2}^N h^{j-1} \sum_{|I|=j} \prod_{i \in I} D_k^+ a_i(x) \prod_{i\not\in I} a_i(x).
\end{equation}
Using Proposition \ref{prop:finite-diff-properties} (ii) we have
\[D^+_ka_j(x) = D^+_kw(x) + n\delta_{j,k}D^-_jw(x+he_k) + nx_jD^+_kD^-_jw(x),\]
for $x \in [0,1]^n_h$. Invoking Proposition \ref{prop:finite-diff-properties} (i) and noting that $D_k^- w(x+he_k) = D^+_kw(x)$, we deduce
\begin{equation}\label{eq:dkaj}
D^+_ka_j(x) = (1 + n\delta_{j,k})\phi(x) + nx_jD^-_j\phi(x).
\end{equation}
Inserting \eqref{eq:z} into  \eqref{eq:dkaj} yields $D^+_ka_j(x_0) \geq \phi(x_0) > 0$ for all $j$. Combining this with \eqref{eq:dk} and recalling that $a_i(x_0) > 0$ gives
\begin{align}\label{eq:dk2}
D^+_kf(x_0) &= \sum_{j=1}^n D^+_k a_j(x_0)\prod_{i\neq j} a_i(x_0) + \sum_{j=2}^N h^{j-1} \sum_{|I|=j} \prod_{i \in I} D_k^+ a_i(x_0) \prod_{i\not\in I} a_i(x_0)\notag\\
& \geq \phi(x_0) \sum_{j=1}^n \prod_{i\neq j} a_i(x_0).
\end{align}
Using the inequality of geometric and arithmetic means yields
\begin{equation}\label{eq:dk3}
\sum_{j=1}^n \prod_{i\neq j} a_i(x_0)\geq n\left( \prod_{j=1}^n \prod_{i\neq j}a_i(x_0) \right)^\frac{1}{n} = n\left(\prod_{i=1}^n  a_i(x_0)\right)^\frac{n-1}{n} = nf(x_0)^\frac{n-1}{n}.
\end{equation}
Combining \eqref{eq:dk2} and \eqref{eq:dk3} establishes \eqref{eq:b1}.

2. We now show that
\begin{equation}\label{eq:b2}
\inf_{U_k} D^+_kw \geq \inf_{x \in U_k} \frac{1}{n} f(x+he_k)^{\frac{1}{n}-1} (D^+_kf(x))_-.
\end{equation}
We may assume that $\phi$ assumes negative values, otherwise \eqref{eq:b2} is trivial. As in the first part of the proof, $\phi$ attains its negative minimum at some $x_0 \in U_k$. Therefore
\begin{equation}\label{eq:z2}
\phi(x_0) <0 \ \text{ and } \  D^-_j \phi(x_0) \leq 0  \ \text{ for all } j \in \{1,\dots,n\}.
\end{equation}   
By \eqref{eq:dkaj}, $D^+_ka_j(x_0)\leq \phi(x_0)<0$ for all $j$. Combining this with Lemma \ref{lem:prod-rule} yields
\begin{align*}\label{eq:dk_2}
D^+_kf(x) &= \sum_{j=1}^n D^+_k a_j(x)\prod_{i\neq j} a_i(x+he_k) - \sum_{j=2}^N h^{j-1} \sum_{|I|=j} \prod_{i \in I} (-D_k^+ a_i(x)) \prod_{i\not\in I} a_i(x+he_k)\\
& \leq \phi(x_0) \sum_{j=1}^n \prod_{i\neq j} a_i(x_0 + he_k)\\
&\leq n\phi(x_0)\left(\prod_{j=1}^n\prod_{i\neq j} a_i(x_0 + he_k)\right)^\frac{1}{n} = nf(x_0+he_k)^\frac{n-1}{n}\phi(x_0).
\end{align*}
This completes the proof.
\end{proof}
\begin{remark}\label{rem:wlip}
Since $w_h \to w$ uniformly on $[0,1]^n$, it is an easy consequence of Theorem \ref{thm:lip} that $w \in C^{0,1}([0,1]^n)$ and
\begin{equation}\label{eq:holder}
[w]_{1;[0,1]^n} \leq \frac{1}{\sqrt{n}} [f]_{1;[0,1]^n}\sup_{[0,1]^n} f^{\frac{1-n}{n}},
\end{equation}
whenever $f \in C^{0,1}([0,1]^n)$ is positive. A sharper result is contained in the following lemma.
\end{remark}
\begin{lemma}\label{lem:wlip}
Let $f$ be a nonnegative function for which $f^\frac{1}{n} \in C^{0,1}([0,1]^n)$, and let $w$ be the maximal bounded viscosity solution of \eqref{P3}. Then $w \in C^{0,1}([0,1]^n)$ and 
\begin{equation}\label{eq:wlip}
[w]_{1;[0,1]^n} \leq \sqrt{n}[f^\frac{1}{n}]_{1;[0,1]^n}.
\end{equation}
\end{lemma}
\begin{proof}
We first assume that $f \in C^2([0,1]^n)$ and $f>0$. For $h>0$, let $w_h$ be the maximal solution of \eqref{S3}. Since $f \in C^2([0,1]^n)$, there exists a constant $C>0$ such that
\[D^+_if(x) \leq f_{x_i}(x) + Ch,\]
for all $i$ and all  $x \in U_i :=\{x \in [0,1]^n_h \, : \, x_i \leq 1-h\}$. By Theorem \ref{thm:lip}
\begin{align*}
D^+_iw_h(x) &\leq \sup_{x \in U_i} \frac{1}{n} f(x)^{\frac{1}{n}-1}(D^+_if(x))_+\\
&\leq \sup_{x \in U_i} \left( \frac{1}{n}f(x)^{\frac{1}{n}-1}f_{x_i}(x) + Ch\right)_+\\
&=\sup_{U_i} \left(  (f^\frac{1}{n})_{x_i} + Ch\right)_+\\
&\leq [f^\frac{1}{n}]_{1;[0,1]^n} + Ch.
\end{align*}
The opposite inequality is obtained similarly, and we find that
\[|D^+_iw_h(x)| \leq [f^\frac{1}{n}]_{1;[0,1]^n} + Ch,\]
for all $i$ and all $x \in U_i$. The estimate \eqref{eq:wlip} follows from the uniform convergence $w_h \to w$ as $h\to 0$.

We now suppose that $f$ is nonnegative and $f^\frac{1}{n} \in C^{0,1}([0,1]^n)$. Extend $f$ to a function on $\R^n$ by setting $f(x) = f(\pi(x))$ for $x \not\in [0,1]^n$, where $\pi:\R^n \to [0,1]^n$ is the closest point projection. Since $\pi$ is 1-Lipschitz, $[f^\frac{1}{n}]_{1,\R^n} = [f^\frac{1}{n}]_{1,[0,1]^n}$. Let $\eta^\eps$ denote the standard mollifier, and define
\[f^\eps = \left( \eta^\eps * f^\frac{1}{n} + \eps\right)^n.\]
It is easy to verify that $f^\eps \in C^\infty([0,1]^n)$,  $f^\eps \to f$ uniformly on $[0,1]^n$ as $\eps\to 0$, and 
\begin{equation}\label{eq:moll}
[(f^\eps)^\frac{1}{n}]_{1;[0,1]^n} \leq [f^\frac{1}{n}]_{1;[0,1]^n}.
\end{equation}
Let $w^\eps$ denote the maximal bounded viscosity solution of \eqref{P3} corresponding to $f^\eps$. Since $f^\eps$ is smooth and positive, the argument above yields
\[[w^\eps]_{1;[0,1]^n} \leq \sqrt{n}[(f^\eps)^\frac{1}{n}]_{1;[0,1]^n} \stackrel{\eqref{eq:moll}}{\leq}\sqrt{n} [f^\frac{1}{n}]_{1;[0,1]^n}.\]
By the comparison principle for \eqref{P3} (Theorem \ref{thm:comparison-w}) and standard results on viscosity solutions~\cite{crandall1992}, we have $w^\eps \to w$ uniformly on $[0,1]^n$ as $\eps \to 0$. Therefore 
\[[w]_{1;[0,1]^n} \leq \sqrt{n} [f^\frac{1}{n}]_{1;[0,1]^n},\]
which completes the proof.
\end{proof}
\begin{remark}\label{rem:tight}
The regularity result in Lemma \ref{lem:wlip} is tight, and cannot be significantly generalized. For intance, if $f(x) =2x_1$, then $f$ is smooth, but $f(x)^\frac{1}{n}= (2x_1)^\frac{1}{n}$ is not Lipschitz on $[0,1]^n$. The corresponding solution of \eqref{P3} is $w(x) = x_1^\frac{1}{n}$, which also fails to be Lipschitz on $[0,1]^n$. Note that the solution of \eqref{P2} is $v(x) = (x_1\cdots x_n)w(x)^n = x_1^2x_2\cdots x_n$, which is smooth on $[0,1]^n$. 
\end{remark}
We can now extend Lipschitz regularity to viscosity solutions of \eqref{P2}.
\begin{lemma}\label{lem:vexun}
Suppose $f$ is a nonnegative function for which $f^\frac{1}{n} \in C^{0,1}([0,1]^n)$, and let $v$ be the maximal viscosity solution of \eqref{P2}. Then $v \in C^{0,1}([0,1]^n)$ and 
\begin{equation}\label{eq:vholder}
[v]_{1;[0,1]^n} \leq C \|f^\frac{n-1}{n}\|_{L^\infty([0,1]^n)}\|f^\frac{1}{n}\|_{C^{0,1}([0,1]^n)}.
\end{equation}
\end{lemma}
\begin{proof}
Notice that we can write
\[v(x) = x_1\cdots x_n w(x)^n, \ \ \ x \in [0,1]^n,\]
where $w$ is the maximal bounded viscosity solution of \eqref{P3}. Therefore
\[[v]_{1;[0,1]^n} \leq [w^n]_{1;[0,1]^n} + [x_1\cdots x_n]_{1;[0,1]^n}\sup_{[0,1]^n} f.\]
The estimate \eqref{eq:vholder} follows from Lemma \ref{lem:wlip} and the inequality
\[[w^n]_{1;[0,1]^n} \leq n\|w\|^{n-1}_{L^\infty([0,1]^n)}[w]_{1;[0,1]^n}.\qedhere\]
\end{proof}
We expect Lemma \ref{lem:vexun} to hold under the weaker condition that $f \in C^{0,1}([0,1]^n)$, but we do not currently know how to prove this.

\subsubsection{Convergence rate}
\label{sec:ratesub}

Since we know the scheme \eqref{S3} converges, we can prove a result analogous to Proposition \ref{prop:extension} regarding extensions of viscosity solutions of \eqref{P3}.
\begin{proposition}\label{prop:extension2}
Suppose $f \in C([0,1]^n)$ is nonnegative and let $w$ be the maximal bounded viscosity solution of \eqref{P3}. Then $\bar{w}(x) = w(x_+)$ is a viscosity solution of
\begin{equation}\label{eq:extension2}
H(x_+,\bar{w},D\bar{w}) =  \bar{f}\ \ \text{ on } (-\infty,1]^n,
\end{equation}
where $\bar{f}(x) = f(x_+)$.
\end{proposition}
\begin{proof}
Let $w_h:[0,1]^n_h \to\R_+$ denote the solution of \eqref{S3} for $h>0$. By Theorem \ref{thm:convS3}, $w_h \to w$ uniformly on $[0,1]^n$. Set $\bar{w}_h(x) = w_h(x_+)$ for $x \in (-\infty,1]^n_h$. By Proposition \ref{prop:extension}, $\bar{w}_h$ is a solution of the discrete scheme    
\[H(x_+,\bar{w}_h(x),D^-\bar{w}_h(x))  = \bar{f}(x) \ \ \text{for all } x \in (-\infty,1]^n_h.\]
Since $\bar{w}_h \to \bar{w}$ uniformly on $(-\infty,1]^n$, Theorem \ref{thm:bs-conv} shows that $\bar{w}$ is a viscosity solution of \eqref{eq:extension2}.
\end{proof}

Before proving a convergence rate for \eqref{S3}, we need another preliminary proposition.
\begin{proposition}\label{prop:H1}
For $(x_0,z_0,p_0) \in \R^n\times \R \times\R^n$ such that $H(x_0,z_0,p_0)>0$ 
\begin{equation}\label{eq:slope}
\frac{\partial H}{\partial z} (x_0,z_0,p_0) \geq nH(x_0,z_0,p_0)^\frac{n-1}{n}.
\end{equation}
\end{proposition}
\begin{proof}
Since $H(x_0,z_0,p_0)>0$ we can write
\[H(x,z,p) = \prod_{i=1}^n (z + nx_i p_i),\]
in a sufficiently small neighborhood of $(x_0,z_0,p_0)$. In this neighborhood
\begin{align*}
\frac{\partial}{\partial z} H(x,z,p) &= H(x,z,p)\sum_{i=1}^n \frac{1}{z + nx_ip_i} \\
&\geq nH(x,z,p)\left( \prod_{i=1}^n \frac{1}{z + nx_ip_i}\right)^\frac{1}{n} \\
&= nH(x,z,p)^\frac{n-1}{n},
\end{align*}
where we used the inequality of arithmetic and geometric means. 
\end{proof}

We now give the proof of Theorem \ref{thm:rate1}. 
\begin{proof}
We extend $f$ to a function on $(-\infty,1]^n$ by setting $f(x) = f(x_+)$.  
By Proposition \ref{prop:extension} we can extend $w_h$ to a function on $(-\infty,1]^n_h$ satisfying 
\[H(x_+,w_h(x),D^-w_h(x))= f(x) \ \ \text{for all } x \in (-\infty,1]^n_h,\]
by setting $w_h(x) = w_h(x_+)$. Similarly, by Proposition \ref{prop:extension2} we can extend $w$ to a function on $(-\infty,1]^n$ by setting $w(x) = w(x_+)$, and $w$ is a viscosity solution of
\[H(x_+,w,Dw) = f \ \ \text{ in } (-\infty,1]^n.\]

We will show that
\begin{equation}\label{eq:halfrate}
w_h - w \leq C\sqrt{h}.
\end{equation}
The opposite inequality is proved similarly.
We can assume that $\sup_{[0,1]^n_h} (w_h - w)>0$. For $\alpha>0$, $x \in (-\infty,1]^n_h$ and $y \in (-\infty,1]^n$, set 
\[\Phi(x,y) = w_h(x) - w(y) - \frac{\alpha}{2}|x-y|^2.\]
Since $w_h(x) = w_h(x_+)$, $w(y) = w(y_+)$ and $|x_+ - y_+| \leq |x-y|$, there exists $x_\alpha \in [0,1]^n_h$ and $y_\alpha \in [0,1]^n$ such that 
\begin{equation}\label{eq:maxcond}
\Phi(x_\alpha,y_\alpha) = \max_{(-\infty,1]^n_h \times (-\infty,1]^n} \Phi.
\end{equation}

Let $\phi(x) = \frac{\alpha}{2}|x-y_\alpha|$. Then $w_h - \phi$ attains its maximum at $x_\alpha$ with respect to the grid $(-\infty,1]^n_h$. Therefore $D^-_iw_h(x_\alpha) \geq D^-_i\phi(x_\alpha)$ for all $i$. Since $x_\alpha \in [0,1]^n_h$ and $H$ is directed, 
\begin{equation}\label{eq:monotonicity}
H(x_\alpha,w_h(x_\alpha),D^-\phi(x_\alpha)) \leq H(x_\alpha,w_h(x_\alpha),D^-w_h(x_\alpha)) = f(x_\alpha).
\end{equation}
Since $\Phi(x_\alpha,x_\alpha) \leq \Phi(x_\alpha,y_\alpha)$, we see that
\begin{equation}\label{eq:pbound}
\frac{\alpha}{2}|x_\alpha - y_\alpha|^2 \leq w(x_\alpha) - w(y_\alpha) \leq C|x_\alpha - y_\alpha|,
\end{equation}
where we invoked Remark \ref{rem:wlip} in the last inequality. Since
\[D^- \phi(x_\alpha) = \alpha(x_\alpha - y_\alpha) - \frac{\alpha h}{2}\vb{1},\]
we have
\[|D^- \phi(x_\alpha) - p_\alpha| \leq C\alpha h,\]
where $p_\alpha = \alpha(x_\alpha - y_\alpha)$. By \eqref{eq:pbound}, $|p_\alpha| \leq C$. Therefore, we can invoke the local Lipschitzness of $H$ in all variables to obtain
\[|H(x_\alpha,w_h(x_\alpha),D^-\phi(x_\alpha)) - H(y_\alpha,w_h(x_\alpha),p_\alpha)| \leq C\left( \alpha h + \frac{1}{\alpha}\right).\]
Combining this with \eqref{eq:monotonicity} yields
\[H(y_\alpha,w_h(x_\alpha),p_\alpha) \leq f(y_\alpha) + C\left(\alpha h + \frac{1}{\alpha}\right).\]
By \eqref{eq:maxcond} we also have
\[H(y_\alpha,w(y_\alpha),p_\alpha) \geq f(y_\alpha).\]
Subtracting these equations yields
\[H(y_\alpha,w_h(x_\alpha),p_\alpha)-H(y_\alpha,w(y_\alpha),p_\alpha) \leq C\left(\alpha h + \frac{1}{\alpha}\right).\]
Since 
\[w_h(x_\alpha) - w(y_\alpha) \geq \Phi(x_\alpha,y_\alpha) \geq \sup_{[0,1]^n_h} (w_h - w)>0,\]
we have $w_h(x_\alpha) > w(y_\alpha)$.  Since $f(y_\alpha)>0$, we can invoke Proposition \ref{prop:H1} to find that 
\[nf(y_\alpha)^\frac{n-1}{n} (w_h(x_\alpha) - w(y_\alpha)) \leq H(y_\alpha,w_h(x_\alpha),p_\alpha)-H(y_\alpha,w(y_\alpha),p_\alpha) \leq  C\left(\alpha h + \frac{1}{\alpha}\right).\]
Choosing $\alpha = \frac{1}{\sqrt{h}}$ we have
\[\sup_{[0,1]^n_h} (w_h - w) \leq w_h(x_\alpha) - w(y_\alpha) \leq C\sqrt{h}. \qedhere\]
\end{proof}

\subsection{The scheme (S2)}
\label{sec:ratesS2}

We now prove a convergence rate for \eqref{S2}. This follows directly from Theorem \ref{thm:rate1} and the following result.
\begin{lemma}\label{lem:changeofvar}
Suppose $f \in C^{0,1}([0,1]^n)$ is positive. Let $w_h$ be the solution of \eqref{S3}, and let $v_h$ be the maximal solution of \eqref{S2}. Then 
\begin{equation}\label{eq:changeofvar}
|x_1\cdots x_n w_h(x)^n - v_h(x)| \leq Ch  \ \ \text{ for all } x \in [0,1]^n_h,
\end{equation}
where
\[C = C\Big(n,[f]_{1;[0,1]^n}, \inf_{[0,1]^n}f\Big).\]
\end{lemma}
\begin{proof}
Set $v(x) = x_1\cdots x_n w_h(x)^n$ for $x \in [0,1]^n_h$. By Lemma \ref{lem:prod-rule}, Remark \ref{rem:wlip}, and the inequality $w_h \geq \inf f^\frac{1}{n}>0$, we have
\[D^-_iv(x) = \frac{x_1\cdots x_n}{x_i} w_h(x)^{n-1}\left( w_h(x) + nx_i D^-_iw_h(x) + O(h)\right).\]
Here, we use the notation $O(h)$ to denote function of $x$ that is bounded uniformly by $Ch$ with
\[C = C\Big(n,[f]_{1;[0,1]^n},\inf_{[0,1]^n}f\Big).\]
Since $w_h$ satisfies \eqref{S3} we have
\[\prod_{i=1}^n \left( w_h(x) + nx_i D^-_iw_h(x) + O(h)\right)_+ = f(x) + O(h).\]
It follows that
\[(D^-_1v(x))_+ \cdots (D^-_nv(x))_+ = v(x)^{n-1}\left(f(x) + O(h)\right).\]

Therefore, there exists a constant $C$ such that
\[(D^-_1v(x))_+ \cdots (D^-_nv(x))_+ \leq v(x)^{n-1}\left(f(x) + Ch\right).\]
Let $\alpha = Ch/\inf f$ and set $\bar{v}_h = (1+\alpha)v_h$.  Then 
\[(D^-_1\bar{v}_h(x))_+ \cdots (D^-_n\bar{v}_h(x))_+  = \bar{v}_h(x)^{n-1}(1+\alpha)f(x) \geq \bar{v}_h(x)^{n-1}(f(x) + Ch).\]
Since $v_h>0$ and $f>0$, we can use Theorem \ref{thm:Svcomp} and Remark \ref{rem:poscomp} to conclude that $v \leq \bar{v}_h\leq v_h + Ch$ on $[0,1]^n_h$.

Likewise, there exists a constant $C$ such that
\[(D^-_1v(x))_+ \cdots (D^-v_n(x))_+ \geq v(x)^{n-1}\left(f(x) - Ch\right).\]
Let $\alpha= Ch/\inf f$, and take $h>0$ small enough so that $\alpha < 1$. Define $\underline{v}_h = (1-\alpha)v_h$. Then 
\[(D^-_1\underline{v}_h(x))_+ \cdots (D^-_n\underline{v}_h(x))_+  = \underline{v}_h(x)^{n-1}(1-\alpha)f(x) \leq \bar{v}_h(x)^{n-1}(f(x) - Ch).\]
Since $v>0$ on $(0,1]^n_h$ we have $v \geq \underline{v}_h \geq v_h - Ch$ on $[0,1]^n_h$. 
\end{proof}
We now have the proof of Theorem \ref{thm:ratev}.
\begin{proof}
Let $w_h$ be the solution of \eqref{S3} and let $w$ be the maximal bounded viscosity solution of \eqref{P3}. By Theorem \ref{thm:rate1}, 
\[|w_h(x) - w(x)| \leq C\sqrt{h} \ \ \text{ for all } x \in [0,1]^n_h.\]
By Lemma \ref{lem:changeofvar} we have
\[|v_h(x) - v(x)| \leq Ch + x_1\cdots x_n|w_h(x)^n - w(x)^n| \leq C\sqrt{h},\]
for all $x \in [0,1]^n_h$ and $h>0$ sufficiently small.  The inequality \eqref{eq:ratev2} follows from the lower bounds
 \[v_h(x),v(x) \geq \delta^n \inf_{[0,1]^n} f \ \text{ for all } x \in [\delta,1]^n_h.\qedhere\]
\end{proof}

\section{Numerical experiments}
\label{sec:numerics}

In this section, we discuss the bisection search method for solving \eqref{S1}--\eqref{S3}, and we show the results of numerical simulations comparing the three schemes. 

\subsection{The bisection search}
\label{sec:bisection}

In dimension $n=2$, all three schemes have a closed form solution (see \eqref{eq:S1closed}, \eqref{eq:S2closed}, and \eqref{eq:S3closed}), and can therefore be efficiently solved up to machine precision. In dimensions $n\geq 3$, it is necessary to use an iterative method to approximate the solution of the scheme. Due to the monotonicity properties of all three schemes, it is natural and efficient to use the  bisection method, or slight variations thereof. The bisection method requires an initial interval to start the search, and a specified tolerance for terminating the bisections. The main idea is to set the tolerance to match the truncation error of each scheme, so that the error from the bisection method is no greater than the error incurred by using finite differences. We describe the details for each scheme below.  

For the scheme (S1), we use the bisection search method at each $x \in (0,1]^n_h$ to construct an approximation $\tilde{u}_h$ of the solution $u_h$ of \eqref{S1}. For the initial interval, we use
\[ I_1= [\max\{\tilde{u}_h(x-he_1),\dots,\tilde{u}_h(x-he_n)\},\max\{\tilde{u}_h(x-he_1),\dots,\tilde{u}_h(x-he_n)\}+ hf(x)^\frac{1}{n}],\]
and with a slight modification of the bisection method, we can ensure that the solution $\tilde{u}_h$ satisfies
\begin{equation}\label{eq:S1bisection}
f(x) \leq D^-_1\tilde{u}_h(x) \cdots D^-_n\tilde{u}_h(x) \leq (1+h)f(x) \ \ \text{ for all } x \in (0,1]^n_h,
\end{equation}
and $\tilde{u}_h(x) = 0$ for $x \in \Gamma_h$.
By comparison for \eqref{S1}  we have
\begin{equation}\label{eq:tildeuh}
u_h(x) \leq \tilde{u}_h(x) \leq (1+h)^\frac{1}{n} u_h(x) \leq u_h(x) +  (x_1\cdots x_n)^\frac{1}{n}(\sup f^\frac{1}{n})h \ \ \text{ for all } x \in (0,1]^n_h,
\end{equation}
where $u_h$ is the exact solution of \eqref{S1}. Since this error is no worse than the truncation error from using first order finite differences, it is unnecessary to continue the bisection search once \eqref{eq:S1bisection} is satisfied. 

For the scheme (S2), we use the same bisection search technique to find $\tilde{v}_h$ satisfying
\[\tilde{v}_h(x)^{n-1}f(x) \leq D^-_1\tilde{v}_h(x) \cdots D^-_n\tilde{v}_h(x) = \tilde{v}_h(x)^{n-1} (1 + h)f(x) \ \  \text{ for all } x \in (0,1]^n_h,\]
with $\tilde{v}_h(x) = 0$ for $x \in \Gamma_h$. By Remark \ref{rem:lipS}, we may take the initial interval to be 
\[ I_2= [\max\{\tilde{v}_h(x-he_1),\dots,\tilde{v}_h(x-he_n)\},\tilde{v}_h(x-he_1)+\cdots+\tilde{v}_h(x-he_n)+ h^nf(x)].\]
Letting $v_h$ denote the exact solution of \eqref{S3} we have by Theorem \ref{thm:Svcomp} and Remark \ref{rem:S2exact} that
\begin{equation}\label{eq:tildevh}
v_h(x) \leq \tilde{v}_h(x) \leq (1+h)v_h(x) \leq v_h(x) + (x_1\cdots x_n)\left(\sup f\right)h.
\end{equation}
As before, this error is smaller than the truncation error introduced by discretizing \eqref{P2}, which justifies our choice of stopping condition for the bisection search.

Finally, for the scheme \eqref{S3}, we use the bisection method to obtain $\tilde{w}_h$ satisfying
\[f(x) \leq \prod_{i=1}^n (\tilde{w}_h(x) + nx_i D^-_i\tilde{w}_h(x)) \leq (1+h)f(x) \ \ \text{ for all } x \in [0,1]^n_h.\]
We take the initial interval for the search to be
\[I_3 = \left[\sigma,\sigma + hf(x)^\frac{1}{n}\prod_{i=1}^n(nx_i + h)^{-\frac{1}{n}}\right],\]
where
\[\sigma = \max\left\{\frac{nx_1\tilde{w}_h(x-he_1)}{nx_1 + h},\dots, \frac{nx_n\tilde{w}_h(x-he_n)}{nx_n + h}\right\}.\]
By Lemma \ref{lem:S3comp} we have
\begin{equation}\label{eq:tildewh}
w_h(x) \leq \tilde{w}_h(x) \leq (1+h)^\frac{1}{n}w_h(x) \leq w_h(x) + \frac{1}{n}\left(\sup f^\frac{1}{n}\right)h,
\end{equation}
where $w_h$ is the exact solution of \eqref{S3}.

\subsection{Simulations}
\label{sec:sim}

\begin{figure}
\centering
\subfigure[$u_1$]{\includegraphics[clip=true,trim = 55 30 50 25, width=0.32\textwidth]{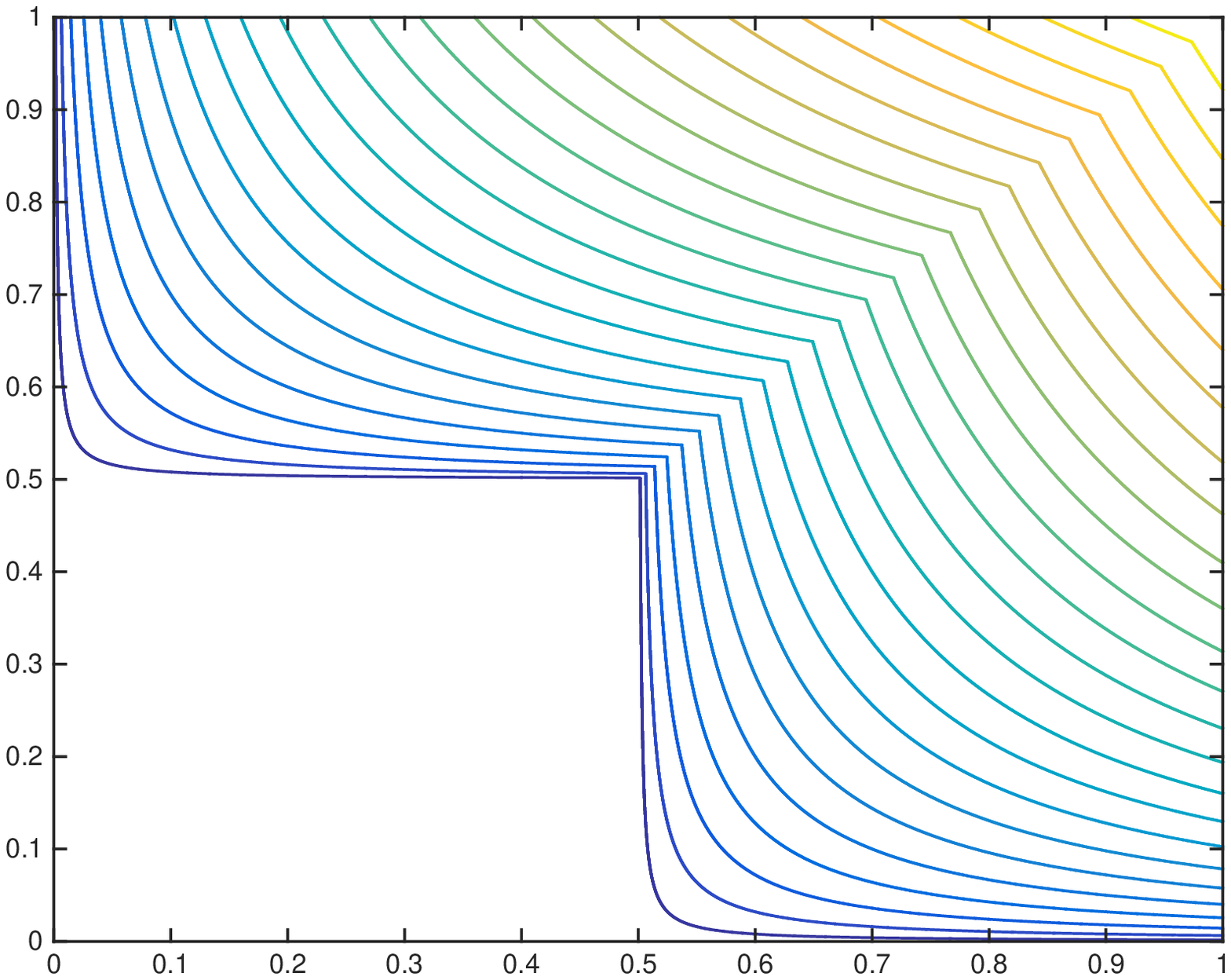}\label{fig:demo1}}
\subfigure[$u_2$]{\includegraphics[clip=true,trim = 55 30 50 25, width=0.32\textwidth]{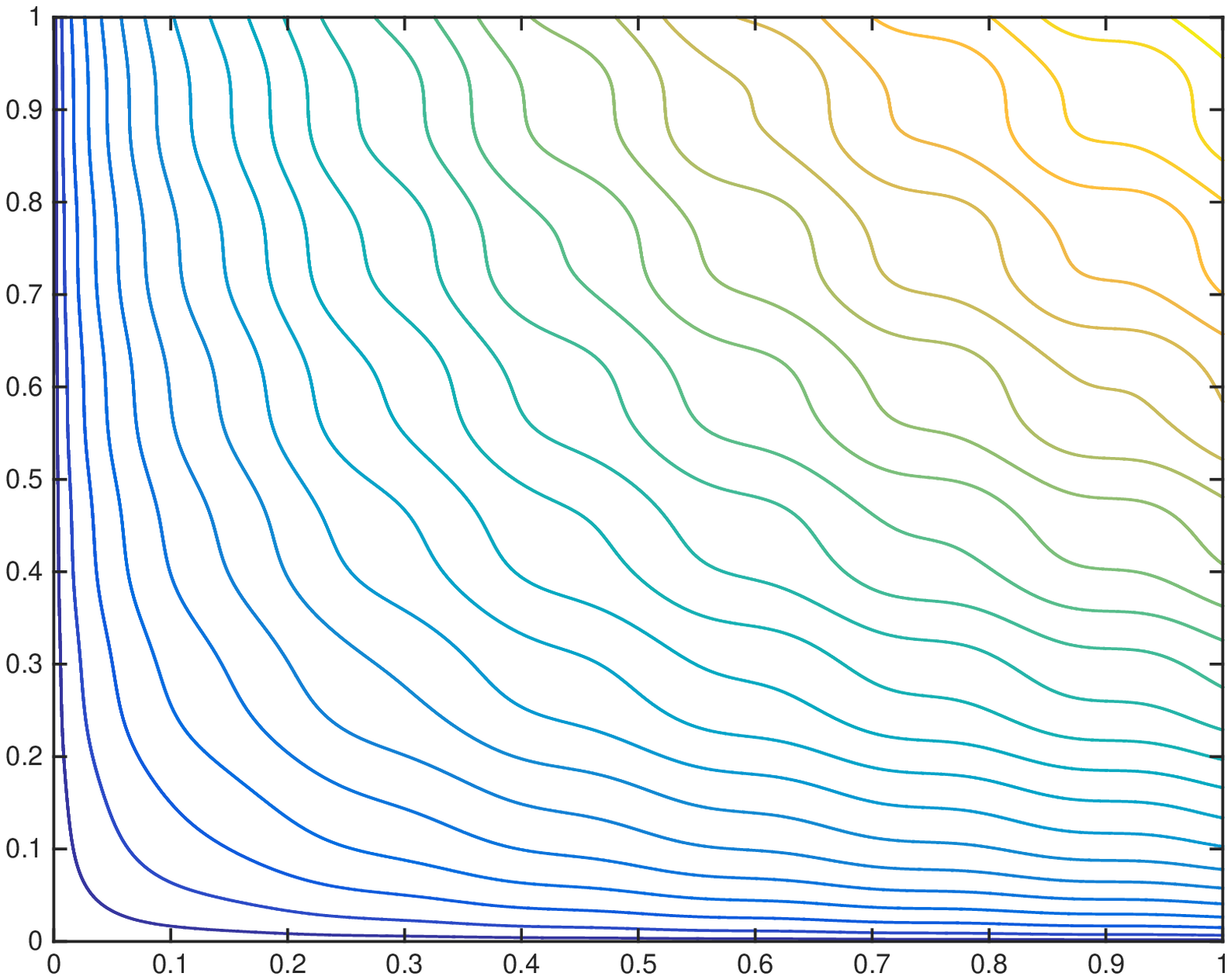}\label{fig:demo2}}
\subfigure[$u_3$]{\includegraphics[clip=true,trim = 55 30 50 25, width=0.32\textwidth]{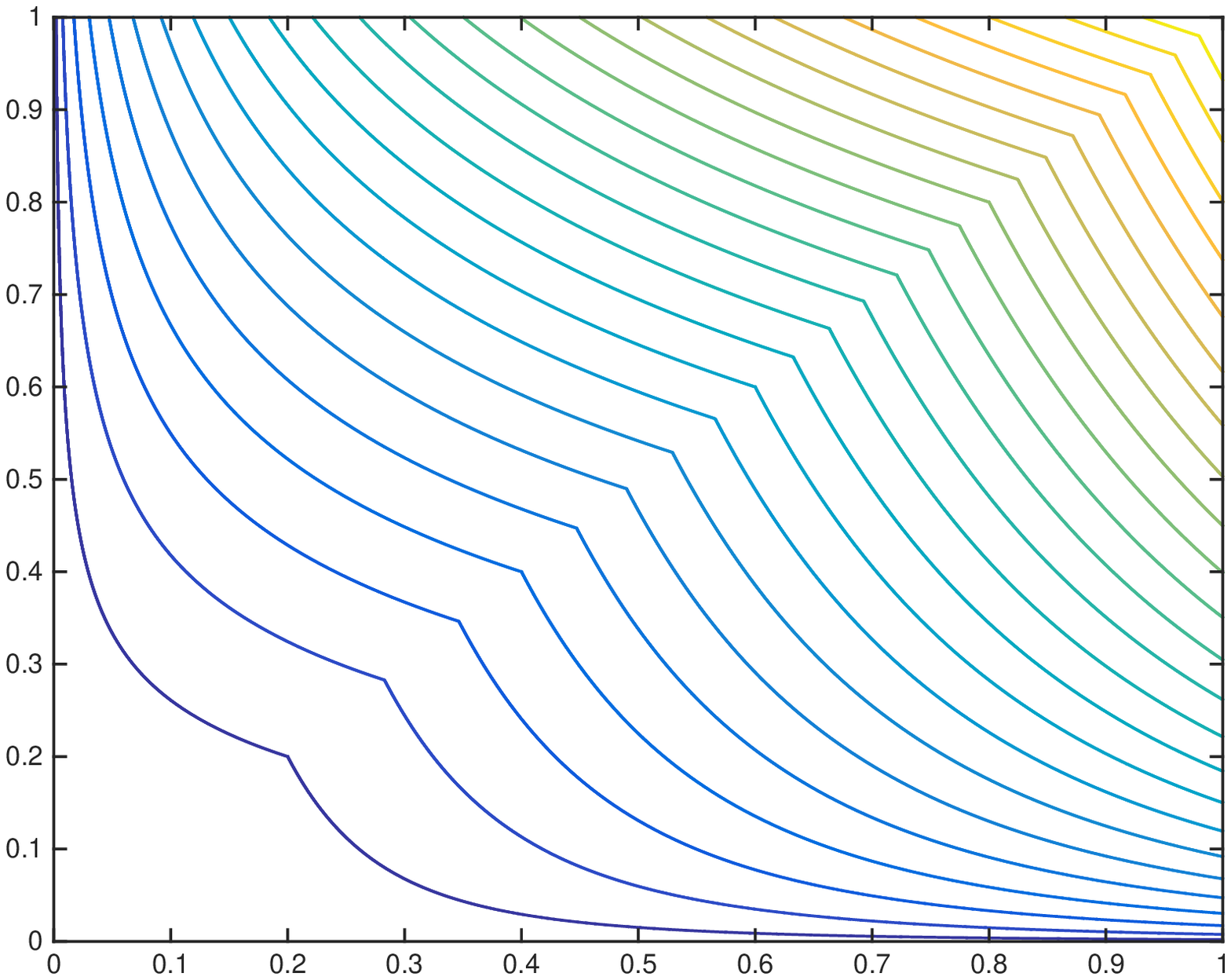}\label{fig:demo3}}
\caption{Depiction of the level sets of the viscosity solutions of \eqref{P1} in dimension $n=2$ for each of the three test cases.}
\label{fig:demou}
\end{figure}

We show here the results of some numerical simulations comparing the three schemes \eqref{S1}, \eqref{S2}, and \eqref{S3}. We first consider
\[f_1(x) = \begin{cases}
1,& \text{if } \max\{x_1,\dots,x_n\}>0.5\\
0,& \text{otherwise}.
\end{cases}\]
The solution of \eqref{P1} is given by
\[u_1(x) = n\max_{i \in \{1,\dots,n\}}\left\{\left(x_i - \frac{1}{2}\right)_+\prod_{j\neq i} x_j\right\}^\frac{1}{n}.\]
The level sets of $u_1$ are illustrated in Figure \ref{fig:demo1}.
Tables \ref{f1n2}, \ref{f1n3} and \ref{f1n4} show the $\ell^\infty$ errors and orders of convergence for $f_1$ in dimensions $n=2$, $n=3$, and $n=4$, respectively. For each scheme, the $\ell^\infty$ errors are measured by how well the schemes approximate the viscosity solution $u_1$ of \eqref{P1}. 
\begin{table}[!t]
\centering
\begin{tabular}{|c|c|c|c|c|c|c|}
 \hline
 &\multicolumn{2}{c|}{(S1)}&\multicolumn{2}{c|}{(S2)}&\multicolumn{2}{c|}{(S3)}\\
\hline
Mesh size $h$& $\ell^\infty$ Error & Order& $\ell^\infty$ Error & Order& $\ell^\infty$ Error & Order\\
\hline
$2.5\times10^{-2}$&$7.1\t{2}$ &         &$2.1\t{2}$ &        &$6.7\t{2}$ & \\
$6.3\times10^{-3}$&$3.4\t{2}$ &$0.54$   &$5.7\t{3}$ &$0.93$  &$3.3\t{2}$ &$0.51$\\
$1.6\times10^{-3}$&$1.6\t{2}$ &$0.51$   &$1.5\t{3}$ &$0.97$  &$1.6\t{2}$ &$0.50$\\
$3.9\times10^{-4}$&$8.2\t{3}$ &$0.50$   &$3.8\t{4}$ &$0.98$  &$8.2\t{3}$ &$0.50$\\
$9.8\times10^{-5}$&$4.1\t{3}$ &$0.50$   &$9.7\t{5}$ &$0.99$  &$4.1\t{3}$ &$0.50$\\
$2.4\times10^{-5}$&$2.0\t{3}$ &$0.50$   &$2.4\t{5}$ &$1.00$  &$2.0\t{3}$ &$0.50$\\
\hline
\end{tabular}
\caption{Rates of convergence for $f_1$ in dimension $n=2$.}
\label{f1n2}
\end{table}
\begin{table}[!t]
\centering
\begin{tabular}{|c|c|c|c|c|c|c|}
 \hline
 &\multicolumn{2}{c|}{(S1)}&\multicolumn{2}{c|}{(S2)}&\multicolumn{2}{c|}{(S3)}\\
\hline
Mesh size $h$& $\ell^\infty$ Error & Order& $\ell^\infty$ Error & Order& $\ell^\infty$ Error & Order\\
\hline
$5\times10^{-2}$&$3.1\t{1}$ &         &$9.1\t{2}$ &        &$2.1\t{1}$ & \\
$2.5\times10^{-2}$&$2.3\t{1}$ &$0.41$   &$5.3\t{2}$ &$0.79$  &$1.7\t{1}$ &$0.34$\\
$1.3\times10^{-2}$&$1.8\t{1}$ &$0.38$   &$3.0\t{2}$ &$0.80$  &$1.3\t{1}$ &$0.34$\\
$6.3\times10^{-3}$&$1.4\t{1}$ &$0.36$   &$1.7\t{2}$ &$0.82$  &$1.1\t{1}$ &$0.33$\\
$3.1\times10^{-3}$&$1.1\t{1}$ &$0.35$   &$9.5\t{3}$ &$0.84$  &$8.5\t{2}$ &$0.33$\\
$1.6\times10^{-3}$&$8.5\t{2}$ &$0.34$   &$5.3\t{3}$ &$0.85$  &$6.7\t{2}$ &$0.33$\\
\hline
\end{tabular}
\caption{Rates of convergence for $f_1$ in dimension $n=3$.}
\label{f1n3}
\end{table}
\begin{table}[!t]
\centering
\begin{tabular}{|c|c|c|c|c|c|c|}
 \hline
 &\multicolumn{2}{c|}{(S1)}&\multicolumn{2}{c|}{(S2)}&\multicolumn{2}{c|}{(S3)}\\
\hline
Mesh size $h$& $\ell^\infty$ Error & Order& $\ell^\infty$ Error & Order& $\ell^\infty$ Error & Order\\
\hline
$2.5\times10^{-1}$&$1.1\t{0}$ &         &$3.8\t{1}$ &        &$4.9\t{1}$ & \\
$1.3\times10^{-1}$&$7.9\t{1}$ &$0.46$   &$2.4\t{1}$ &$0.68$  &$4.1\t{1}$ &$0.26$\\
$6.3\times10^{-2}$&$6.0\t{1}$ &$0.40$   &$1.5\t{1}$ &$0.64$  &$3.5\t{1}$ &$0.25$\\
$3.1\times10^{-2}$&$4.7\t{1}$ &$0.36$   &$9.5\t{2}$ &$0.70$  &$2.9\t{1}$ &$0.25$\\
$1.6\times10^{-2}$&$3.7\t{1}$ &$0.32$   &$5.7\t{2}$ &$0.72$  &$2.5\t{1}$ &$0.25$\\
$7.8\times10^{-3}$&$3.1\t{1}$ &$0.29$   &$3.4\t{2}$ &$0.75$  &$2.1\t{1}$ &$0.25$\\
\hline
\end{tabular}
\caption{Rates of convergence for $f_1$ in dimension $n=4$.}
\label{f1n4}
\end{table}

This is an interesting test case because $f_1$ is discontinuous, so Theorems \ref{thm:ratev} and \ref{thm:rate1} do not apply. We see that \eqref{S1} and \eqref{S3} have experimental convergence rates on the order of $O(h^\frac{1}{n})$. This should be expected, as $u_1$ and $w_1 := u_1/n(x_1\cdots x_n)^\frac{1}{n}$ are at most H\"older continuous with exponent $\frac{1}{n}$, and have similar gradient singularities where $u_1$ transitions from zero to a positive value. On the other hand, $v_1 := u_1^n/n^n$ is Lipschitz continuous on $[0,1]^n$, and we correspondingly observe a better convergence rate from \eqref{S2}. This illustrates the important fact that the transformation $v=u^n/n^n$ regularizes gradient singularities \emph{anywhere} in the domain $[0,1]^n$, and therefore we expect \eqref{S2} to have the best convergence rate in general. It is interesting to note that \eqref{S2} attains its formal convergence rate of $O(h)$ in dimension $n=2$, but appears to have a strictly worse rate in higher dimensions. 

The second case we consider is 
\[f_2(x) = \frac{1}{n^n(k+1)^n}\prod_{i=1}^n \left(\sum_{j=1}^n \sin(kx_j)^2 + nk + nkx_i\sin(2kx_i)\right),\]
where $k>0$. In the experiments, we set $k=20$. The solution of \eqref{P1} is
\[u_2(x) = \frac{1}{k+1}(x_1\cdots x_n)^\frac{1}{n}\left(\sum_{j=1}^n \sin(kx_j)^2 + nk\right).\]
The level sets of $u_2$ are shown in Figure \ref{fig:demo2}.
This case is interesting because the solution $u_2$ is smooth on $(0,1]^n$, so we expect to see the formal rates of convergence for each scheme. Tables \ref{f2n2}, \ref{f2n3} and \ref{f2n4} show the $\ell^\infty$ errors and orders of convergence in dimensions $n=2$, $n=3$, and $n=4$, respectively. As expected, schemes \eqref{S2} and \eqref{S3} show $O(h)$ rates, while the rate for \eqref{S1} appears to be approaching $O(h^\frac{1}{n})$ as $n$ increases.
\begin{table}[!t]
\centering
\begin{tabular}{|c|c|c|c|c|c|c|}
 \hline
 &\multicolumn{2}{c|}{(S1)}&\multicolumn{2}{c|}{(S2)}&\multicolumn{2}{c|}{(S3)}\\
\hline
Mesh size $h$& $\ell^\infty$ Error & Order& $\ell^\infty$ Error & Order& $\ell^\infty$ Error & Order\\
\hline
$2.5\times10^{-2}$&$9.5\t{2}$ &         &$2.4\t{2}$ &        &$2.4\t{2}$ & \\
$6.3\times10^{-3}$&$4.6\t{2}$ &$0.53$   &$6.1\t{3}$ &$0.99$  &$5.9\t{3}$ &$1.01$\\
$1.6\times10^{-3}$&$2.3\t{2}$ &$0.50$   &$1.6\t{3}$ &$0.97$  &$1.4\t{3}$ &$1.02$\\
$3.9\times10^{-4}$&$1.1\t{2}$ &$0.50$   &$4.1\t{4}$ &$0.98$  &$3.5\t{4}$ &$1.02$\\
$9.8\times10^{-5}$&$5.6\t{3}$ &$0.50$   &$1.0\t{4}$ &$0.99$  &$8.8\t{5}$ &$1.00$\\
$2.4\times10^{-5}$&$2.8\t{3}$ &$0.50$   &$2.6\t{5}$ &$1.00$  &$2.2\t{5}$ &$0.99$\\
\hline
\end{tabular}
\caption{Rates of convergence for $f_2$ in dimension $n=2$.}
\label{f2n2}
\end{table}
\begin{table}[!t]
\centering
\begin{tabular}{|c|c|c|c|c|c|c|}
 \hline
 &\multicolumn{2}{c|}{(S1)}&\multicolumn{2}{c|}{(S2)}&\multicolumn{2}{c|}{(S3)}\\
\hline
Mesh size $h$& $\ell^\infty$ Error & Order& $\ell^\infty$ Error & Order& $\ell^\infty$ Error & Order\\
\hline
$5\times10^{-2}$&$3.6\t{1}$ &         &$6.6\t{2}$ &        &$5.6\t{2}$ & \\
$2.5\times10^{-2}$&$2.8\t{1}$ &$0.39$   &$4.8\t{2}$ &$0.46$  &$4.0\t{2}$ &$0.48$\\
$1.3\times10^{-2}$&$2.2\t{1}$ &$0.36$   &$2.4\t{2}$ &$1.02$  &$2.0\t{2}$ &$1.01$\\
$6.3\times10^{-3}$&$1.7\t{1}$ &$0.35$   &$1.2\t{2}$ &$0.94$  &$1.0\t{2}$ &$0.96$\\
$3.1\times10^{-3}$&$1.3\t{1}$ &$0.34$   &$6.2\t{3}$ &$0.98$  &$5.3\t{3}$ &$0.97$\\
$1.6\times10^{-3}$&$1.1\t{1}$ &$0.34$   &$3.2\t{3}$ &$0.96$  &$2.7\t{3}$ &$0.96$\\
\hline
\end{tabular}
\caption{Rates of convergence for $f_2$ in dimension $n=3$.}
\label{f2n3}
\end{table}
\begin{table}[!t]
\centering
\begin{tabular}{|c|c|c|c|c|c|c|}
 \hline
 &\multicolumn{2}{c|}{(S1)}&\multicolumn{2}{c|}{(S2)}&\multicolumn{2}{c|}{(S3)}\\
\hline
Mesh size $h$& $\ell^\infty$ Error & Order& $\ell^\infty$ Error & Order& $\ell^\infty$ Error & Order\\
\hline
$2.5\times10^{-1}$&$1.6\t{0}$ &         &$4.0\t{1}$ &        &$3.1\t{1}$ & \\
$1.3\times10^{-1}$&$1.2\t{0}$ &$0.42$   &$3.9\t{1}$ &        &$3.8\t{1}$ &\\
$6.3\times10^{-2}$&$6.9\t{1}$ &$0.79$   &$1.4\t{1}$ &$1.49$  &$1.1\t{1}$ &$1.81$\\
$3.1\times10^{-2}$&$5.3\t{1}$ &$0.37$   &$7.2\t{2}$ &$0.93$  &$5.8\t{2}$ &$0.88$\\
$1.6\times10^{-2}$&$4.3\t{1}$ &$0.30$   &$3.7\t{2}$ &$0.94$  &$3.0\t{2}$ &$0.96$\\
$7.8\times10^{-3}$&$3.4\t{1}$ &$0.28$   &$1.9\t{2}$ &$0.98$  &$1.5\t{2}$ &$0.96$\\
\hline
\end{tabular}
\caption{Rates of convergence for $f_2$ in dimension $n=4$.}
\label{f2n4}
\end{table}

Finally, we consider a case where $f$ is Lipschitz, and $u$ has a gradient discontinuity, which is common in Hamilton-Jacobi equations due to crossing characteristics. In such a case, the solution $u$ is not smooth, and so it is unclear \emph{a priori} whether we will observe the formal or theoretical rates. An example in this setting is 
\[u_3(x) = n(x_1\cdots x_n)^\frac{1}{n}w_3(x),\]
where
\[w_3(x) =C\max\{x_1,\dots,x_n\} +  \sum_{j=1}^n x_j.\]
The level sets of $u_3$ are shown in Figure \ref{fig:demo3}.
The corresponding right hand side of \eqref{P1} is
\[f_3(x) = \frac{1}{(C+n)^n}\left(w_3(x)  + n(1+C)x(n)\right)\prod_{i=1}^{n-1}\left( w_3(x) + nx(i)\right).\]
where $x(i) = x_{\pi_x(i)}$ for a permutation $\pi_x$ such that $x(1)\leq x(2) \leq\cdots \leq x(n)$. We set $C=10$ in the experiments. Tables \ref{f3n2}, \ref{f3n3} and \ref{f3n4} show the $\ell^\infty$ errors and orders of convergence in dimensions $n=2$, $n=3$, and $n=4$, respectively. We note that the formal rates of convergence are observed in this case for all schemes. 
\begin{table}[!t]
\centering
\begin{tabular}{|c|c|c|c|c|c|c|}
 \hline
 &\multicolumn{2}{c|}{(S1)}&\multicolumn{2}{c|}{(S2)}&\multicolumn{2}{c|}{(S3)}\\
\hline
Mesh size $h$& $\ell^\infty$ Error & Order& $\ell^\infty$ Error & Order& $\ell^\infty$ Error & Order\\
\hline
$2.5\times10^{-2}$&$8.3\t{2}$ &         &$7.5\t{2}$ &        &$3.1\t{2}$ & \\
$6.3\times10^{-3}$&$4.2\t{2}$ &$0.49$   &$1.9\t{2}$ &$1.00$  &$8.0\t{3}$ &$0.98$\\
$1.6\times10^{-3}$&$2.1\t{2}$ &$0.50$   &$4.7\t{3}$ &$1.00$  &$2.0\t{3}$ &$1.00$\\
$3.9\times10^{-4}$&$1.1\t{2}$ &$0.50$   &$1.2\t{3}$ &$1.00$  &$5.0\t{4}$ &$1.00$\\
$9.8\times10^{-5}$&$5.3\t{3}$ &$0.50$   &$2.9\t{4}$ &$1.00$  &$1.3\t{4}$ &$1.00$\\
$2.4\times10^{-5}$&$2.7\t{3}$ &$0.50$   &$7.4\t{5}$ &$1.00$  &$3.1\t{5}$ &$1.00$\\
\hline
\end{tabular}
\caption{Rates of convergence for $f_3$ in dimension $n=2$.}
\label{f3n2}
\end{table}

\begin{table}[!t]
\centering
\begin{tabular}{|c|c|c|c|c|c|c|}
 \hline
 &\multicolumn{2}{c|}{(S1)}&\multicolumn{2}{c|}{(S2)}&\multicolumn{2}{c|}{(S3)}\\
\hline
Mesh size $h$& $\ell^\infty$ Error & Order& $\ell^\infty$ Error & Order& $\ell^\infty$ Error & Order\\
\hline
$5\times10^{-2}$&$3.0\t{1}$ &         &$2.6\t{1}$ &        &$1.3\t{1}$ & \\
$2.5\times10^{-2}$&$2.5\t{1}$ &$0.31$   &$1.3\t{1}$ &$0.96$  &$6.8\t{2}$ &$0.95$\\
$1.3\times10^{-2}$&$2.0\t{1}$ &$0.32$   &$6.7\t{2}$ &$0.99$  &$3.5\t{2}$ &$0.97$\\
$6.3\times10^{-3}$&$1.6\t{1}$ &$0.33$   &$3.3\t{2}$ &$0.99$  &$1.8\t{2}$ &$0.98$\\
$3.1\times10^{-3}$&$1.2\t{1}$ &$0.33$   &$1.7\t{2}$ &$1.00$  &$8.8\t{3}$ &$0.99$\\
$1.6\times10^{-3}$&$9.9\t{2}$ &$0.33$   &$8.4\t{3}$ &$1.00$  &$4.4\t{3}$ &$1.00$\\
\hline
\end{tabular}
\caption{Rates of convergence for $f_3$ in dimension $n=3$.}
\label{f3n3}
\end{table}
\begin{table}[!t]
\centering
\begin{tabular}{|c|c|c|c|c|c|c|}
 \hline
 &\multicolumn{2}{c|}{(S1)}&\multicolumn{2}{c|}{(S2)}&\multicolumn{2}{c|}{(S3)}\\
\hline
Mesh size $h$& $\ell^\infty$ Error & Order& $\ell^\infty$ Error & Order& $\ell^\infty$ Error & Order\\
\hline
$2.5\times10^{-1}$&$8.5\t{1}$ &         &$1.4\t{0}$ &        &$6.3\t{1}$ & \\
$1.3\times10^{-1}$&$6.6\t{1}$ &$0.37$   &$7.7\t{1}$ &$0.84$  &$3.8\t{1}$ &$0.73$\\
$6.3\times10^{-2}$&$5.5\t{1}$ &$0.26$   &$4.1\t{1}$ &$0.89$  &$2.1\t{1}$ &$0.86$\\
$3.1\times10^{-2}$&$4.6\t{1}$ &$0.25$   &$2.2\t{1}$ &$0.94$  &$1.1\t{1}$ &$0.91$\\
$1.6\times10^{-2}$&$3.9\t{1}$ &$0.25$   &$1.1\t{1}$ &$0.98$  &$5.6\t{2}$ &$0.96$\\
$7.8\times10^{-3}$&$3.2\t{1}$ &$0.25$   &$5.5\t{2}$ &$0.99$  &$2.8\t{2}$ &$0.99$\\
\hline
\end{tabular}
\caption{Rates of convergence for $f_3$ in dimension $n=4$.}
\label{f3n4}
\end{table}

\section*{Acknowledgments}
The author is grateful to Selim Esedo\=glu and Lawrence C.~Evans for valuable discussions about this work.
\appendix

\section{Properties of finite differences}
We give the proof of Lemma \ref{lem:prod-rule} here.
\begin{proof}
We proceed by induction. The base case of $N=2$ is exactly Proposition \ref{prop:finite-diff-properties} (ii). Assume \eqref{eq:prod-rule} holds for some $N\geq 2$. Then by the inductive hypothesis and base case
\begin{align}\label{eq:prod1}
D^\pm_k(u_1\cdots u_N u_{N+1}) &= (u_1 \cdots u_N) D^\pm_k u_{N+1} + u_{N+1}D^\pm_k (u_1\dots u_N) \pm hD^\pm_k (u_1\cdots u_N) D^\pm_k u_{N+1}\notag\\
& = \sum_{j=1}^{N+1} D^\pm_k u_j\prod_{i\neq j} u_i  + u_{N+1}\sum_{j=2}^N (\pm h)^{j-1} \sum_{|I|=j} \prod_{i \in I} D_k^\pm u_i \prod_{i\not\in I} u_i\notag \\
&\hspace{1.5in} \pm hD^\pm_k (u_1\cdots u_N) D^\pm_k u_{N+1}.
\end{align}
Notice that
\begin{align}\label{eq:prod2}
\pm hD^\pm_k& (u_1\cdots u_N) D^\pm_k u_{N+1}\notag \\
&= \pm h D^\pm_k u_{N+1}\left(\sum_{j=1}^N D^\pm_k u_j\prod_{i\neq j} u_i  + \sum_{j=2}^N (\pm h)^{j-1} \sum_{|I|=j} \prod_{i \in I} D_k^\pm u_i \prod_{i\not\in I} u_i  \right)\notag \\
&= \sum_{j=1}^N \pm h D^\pm_k u_{N+1}D^\pm_k u_j\prod_{i\neq j} u_i  + \sum_{j=2}^N (\pm h)^j \sum_{|I|=j} D^\pm_k u_{N+1}\prod_{i \in I} D_k^\pm u_i \prod_{i\not\in I} u_i\notag \\
&= \sum_{j=2}^{N+1} (\pm h)^{j-1} \sum_{\substack{|I|=j \\ N+1 \in I}} \prod_{i \in I} D_k^\pm u_i \prod_{i\not\in I} u_i,
\end{align}
where in the final summation, $I \subset \{1,\dots,N+1\}$. 
We also have
\[u_{N+1}\sum_{j=2}^N (\pm h)^{j-1} \sum_{|I|=j} \prod_{i \in I} D_k^\pm u_i \prod_{i\not\in I} u_i = \sum_{j=2}^{N+1} (\pm h)^{j-1} \sum_{\substack{|I|=j \\ N+1 \not\in I}} \prod_{i \in I} D_k^\pm u_i \prod_{i\not\in I} u_i.\]
Combining this with \eqref{eq:prod1} and \eqref{eq:prod2} we have
\[D^\pm_k(u_1\cdots u_N u_{N+1}) = \sum_{j=1}^{N+1}D^\pm_k u_j \prod_{i\neq j} u_i+ \sum_{j=2}^{N+1} (\pm h)^{j-1} \sum_{|I|=j} \prod_{i \in I} D_k^\pm u_i \prod_{i\not\in I} u_i.\]
This verifies \eqref{eq:prod-rule} for $N+1$. The proof is completed by mathematical induction.
\end{proof}

%\bibliography{ref}
%\bibliographystyle{abbrv}
\end{document}